\documentclass[a4paper, 12pt]{amsart}
\usepackage[left=1.5cm,right=1.5cm, top=1.5cm, bottom=2cm]{geometry}
\usepackage[latin1]{inputenc}
\usepackage[T1]{fontenc}
\usepackage{indentfirst}
\usepackage{latexsym}
\usepackage[none]{hyphenat}
\usepackage{amsmath}
\usepackage{amsthm}
\usepackage{amsfonts}
\usepackage{mathrsfs}
\usepackage{amssymb}
\usepackage{wasysym}
\usepackage{enumerate}
\usepackage{color}

\usepackage[centertags]{amsmath}
\usepackage[english]{babel}
\usepackage[active]{srcltx}

\usepackage[colorlinks]{hyperref}

\newcommand{\Real}{\mathbb{R}}

\newcommand{\supp}{\operatorname{supp}}

\newtheorem{Teo}{\textbf{Theorem}}[section]
\newtheorem{Pro}[Teo]{\textbf{{Proposition}}}
\newtheorem{Cor}[Teo]{\textbf{{Corollary}}}
\newtheorem{Def}[Teo]{\textbf{{Definition}}}
\newtheorem{Lem}[Teo]{\textbf{{Lemma}}}

\newtheorem{Exs}[Teo]{\textbf{{Examples}}}

\newtheorem{Rem}[Teo]{\textbf{{Remark}}}

\newcommand{\tresp}[1]{{\left\vert\kern-0.25ex\left\vert\kern-0.25ex\left\vert #1
    \right\vert\kern-0.25ex\right\vert\kern-0.25ex\right\vert}}

\title{A class of Hamilton-Jacobi equations on  Banach-Finsler manifolds}

\author{J.A. Jaramillo, M. Jim\'enez-Sevilla, J.L. R\'odenas-Pedregosa, L. S\'anchez-Gonz\'alez}

\date{September 9th, 2014}

\address{J.A.  Jaramillo, Instituto de Matem\'atica
Interdisciplinar \\ Departamento de An\'alisis Matem\'atico \\ Facultad de Matem{\'a}ticas\\ Universidad Complutense de Madrid
\\  Madrid 28040, Spain.}

\email{jaramil@mat.ucm.es}

\address{M. Jim\'enez-Sevilla, Instituto de Matem\'atica
Interdisciplinar \\ Departamento de An\'alisis Matem\'atico \\ Facultad de Matem{\'a}ticas\\ Universidad Complutense de Madrid
\\  Madrid 28040, Spain.}

\email{marjim@mat.ucm.es}

\address{J.L. R\'odenas-Pedregosa, Departamento de Matem\'atica Aplicada I \\ Escuela T\'ecnica Superior de Ingenieros Industriales \\ Universidad Nacional de Educaci\'on a Distancia \\ Madrid 28040, Spain.}

\email{jlrodenas@ind.uned.es}

\address{L. S\'anchez-Gonz\'alez, Departamento de Ingenier{\'i}a Matem{\'a}tica\\ Facultad de CC. F{\'i}sicas y  Matem{\'a}ticas\\ Universidad de Concepci{\'o}n\\ Casilla 160-C, Concepci{\'o}n, Chile.}

\email{lsanchez@ing-mat.udec.cl}

\thanks{J.A. Jaramillo, M. Jim\'enez-Sevilla and L. S\'anchez-Gonz\'alez have been supported in part by DGES (Spain) Project MTM2012-34341. J.L. R\'odenas-Pedregosa has
   been supported in part by grant MEC BES-2013-066316 and by DGES (Spain) Project MTM2012-30942.
   L. S\'anchez-Gonz\'alez has also been
   supported by CONICYT-Chile through FONDECYT Project 11130354.}



\begin{document}

\maketitle

{\em During the last part of the editorial processing of this article, our dear friend and colleague}

{\em  Luis Sánchez-González passed away unexpectedly. We dedicate this paper to his memory.}

\begin{abstract} The concept of subdifferentiability is studied in the context of $C^1$ Finsler manifolds (modeled on a Banach space with a Lipschitz $C^1$ bump function). A class of Hamilton-Jacobi equations defined on $C^1$ Finsler manifolds is studied and several results related to the existence and uniqueness of viscosity solutions are obtained.
\end{abstract}

\section{Introduction}
This work is mainly devoted to the study of a certain class of Hamilton-Jacobi equations defined on Banach-Finsler manifolds. Along the way, we also develop some techniques of subdifferential calculus which are needed in this context. This paper is a continuation of \cite{JSSG}, where basic properties and a smooth variational principle were studied in the context of Banach-Finsler manifolds. In particular, we apply some of the results obtained in \cite{JSSG}, as well as some techniques studied in the cases of Hamilton-Jacobi equations on $\mathbb R^n$, on Banach spaces and on Riemannian manifolds \cite{I, IShort, DGZ1, D, DG, AFLM, AFLM2}, in order to obtain our results about existence and uniqueness of viscosity solutions of a class of Hamilton-Jacobi equations on Banach-Finsler manifolds.

The concepts of subdifferentiability and viscosity solutions of Hamilton-Jacobi equations have been extensively studied by many authors. The notion of viscosity solution was introduced by M.G. Crandall and P.L. Lions (see for instance \cite{CL1,CL2}). It was H. Ishii who first  introduced the method of Perron to derive the existence of viscosity solutions of Hamilton-Jacobi equations \cite{I}. The literature about this subject is huge. For an introduction we can mention the books by G. Barles \cite{Barles} and by P.L. Lions \cite{L}. For a detailed account and further information, we refer the reader to the recent survey of H. Ishii \cite{IShort} and references therein.

The study of the above mentioned concepts in (finite and infinite dimensional) Riemannian manifolds was introduced by D. Azagra, J. Ferrera and F. L\'opez-Mesas in \cite{AFLM, AFLM2}. Let us also mention the related work of Y.S. Ledyaev and Q.J. Zhu \cite{LZ} who studied subdifferentiability and generalized solutions of first-order partial differential equations on (finite dimensional) Riemannian manifolds.

In this work we attempt to continue the study of subdifferentiability and viscosity solutions of  Hamilton-Jacobi equations in a non-Riemannian setting. In this way we consider the more general context of (finite and infinite dimensional) Finsler manifolds. Our manifolds will be modeled on a Banach space $X$ which admits a $C^1$ Lipschitz bump function, which provides, as we will see, a quite natural setting for the class of Hamilton-Jacobi equations under our consideration.

The contents of the paper are arranged as follows. In the second section, we recall the definitions of  $C^1$ Finsler manifold $M$ modeled over a Banach space,  Finsler metric over the manifold $M$ (in the sense of Palais) and Fr\'{e}chet subdifferentiability of a function $f:M\to (-\infty, \infty]$.  Basic properties of the subdifferential are established, such as: a local fuzzy rule for the subdifferential of the sum, via localizing charts and the corresponding fuzzy rule in Banach spaces (\cite{DEL,AFLM}); the density of the points  of subdifferentiability (in the domain) of a lower semicontinuous function;  and also a  Mean Value inequality for  lower semicontinous functions defined on  Finsler manifolds, in the same vein as the ones obtained by R. Deville \cite{D2} for Banach spaces and D. Azagra, J. Ferrera and F. L\'opez-Mesas \cite{AFLM} for Riemannian manifolds.

In the third section, we study the existence of a unique viscosity solution of the eikonal equation defined on a bounded open subset of a $C^1$ Finsler manifold modeled on a Banach space $X$ with a $C^1$ Lipschitz bump function. The eikonal equation has been largely studied by many authors. In the works of L.A. Caffarelli and M.G. Crandall \cite{CC} and A. Siconolfi \cite{S} the authors consider the construction of a Finsler metric associated to the eikonal equation defined on bounded open subsets of $\mathbb R^n$. Let us also mention the recent work of P. Angulo and L. Guijarro \cite{AG} related to the eikonal equation on bounded open subsets of (finite dimensional) Riemannian manifolds.

In the fourth section, we obtain a comparison and stability result for bounded and locally Lipschitz viscosity solutions of (stationary) Hamilton-Jacobi equations of the form $u(x)+H(x,||du(x)||_x)=0$ for $x\in M$, where the Hamiltonian $H:M\times \mathbb R\to \mathbb R$ satisfies a condition weaker than uniform continuity. Moreover, we determine a result on the
existence of bounded viscosity solutions under  additional conditions such  as
the coercivity of the Hamiltonian. Also, let us recall here the related results of J. Borwein, Q.J. Zhu, R. Deville, G. Godefroy, V. Zizler and E.M. El Haddad \cite{BZ, DGZ1, D, DG, DGZ, EH, EH2} obtained for Banach spaces.

In the  fifth section,  we study a comparison and monotony result for  viscosity solutions of (evolution) Hamilton-Jacobi equations of the form $u_t(t,x)+H(t,x,||u_x(t,x)||_x)=0$ for $(t,x)\in [0,\infty)\times M$ and initial condition $u(0,x)=h(x)$ for $x\in M$, where the Hamiltonian $H:M\times \mathbb R\to \mathbb R$ satisfies a condition weaker  than uniform continuity  and the initial condition $h$ is bounded and continuous.
In order to establish the comparison result, additional conditions on $u$ are required:  $u$ is bounded in $[0,T)\times M$ for every $T>0$ and $u$ is locally Lipschitz.
Also, a result about existence of viscosity solutions (bounded in $[0,T)\times M$ for every $T>0$) is determined
within some specific conditions.

The notation we use is standard. The norm in a Banach space $X$ is denoted by
$||\cdot||$ and the dual norm in the dual Banach space $X^*$ is denoted as
$||\cdot||^*$. We will say that the norms $||\cdot||_{1}$
and   $||\cdot||_{2}$ defined on $X$
 are $K$-equivalent ($K\ge 1$) whether $\frac{1}{K}||v||_1\le||v||_2\le K||v||_1$,
  for every
$v\in X$.  A $C^k$  bump function $b:X\to \Real$
(where $k\in \mathbb N\cup \{\infty\}$) is a $C^k$ smooth function on X with bounded,
 non-empty support, where $\supp(b) =\overline{\{x \in  X : b(x)\neq 0\}}$.
A  function $f:(A,d)\to \mathbb R$, where $(A,d)$ is a metric space, is $L$-Lipschitz (with $L\ge 0$) if
$|f(y)-f(z)| \le L d(y,z)$,  for all $y,z \in A$.
If $M$ is a Banach-Finsler manifold, we denote by $T_xM$ the tangent space of $M$ at
 $x$ and by $T_xM^*$ the dual space
of $T_xM$. Recall that the tangent bundle of $M$ is $TM=\{(x,v):x\in M \text{ and }
v\in T_x M\}$ and the cotangent bundle of $M$ is $TM^*=\{(x,\tau):x\in M \text{ and }
\tau\in T_x M^*\}$. We refer to \cite{Deim} and \cite{JSSG}   for additional
 definitions. For a set $A$, we call a function $f:A\rightarrow (-\infty, \infty]$
 proper when the set $\mathop{dom}{f}:=\{x\in M\,:\,f(x)<+\infty\}$  is nonempty.

\section{Subdifferentials on Banach-Finsler manifolds}

Let us begin with the definition of Finsler manifold in the sense of Palais and some
basic properties.

\begin{Def} \label{defFinsler}
For $\ell \in \mathbb N\cup \{\infty\}$, let $M$ be a (paracompact) $C^\ell$ Banach manifold
 modeled on a Banach space
$(X,||\cdot||)$. Consider $TM$ the tangent  bundle
of $M$ and a continuous  map $||\cdot||_M: TM\to [0,\infty)$. We say that
$(M,||\cdot||_M)$ is a \textbf{$C^\ell$ Finsler manifold in the sense of Palais}
(see \cite{Palais,Deim,Rabier})
if $||\cdot||_M$ satisfies the following conditions:
\begin{enumerate}
\item[(P1)] For every $x\in M$, the map $||\cdot||_x:={||\cdot||_M}|_{T_xM}:
T_xM\to [0,\infty)$ is a norm on the tangent space $T_xM$ such that for every chart
$\varphi:U\to X$ with $x \in U$, the norm
$v\in X \mapsto ||d\varphi^{-1}(\varphi(x))(v)||_x$ is equivalent to $||\cdot||$ on
$X$.
\item[(P2)] For every $x_0\in M$, $\varepsilon>0$ and every chart $\varphi:U\to X$
with $x_0\in U$, there is an open neighborhood $W$ of $x_0$  such that if  $x\in W$
and $v\in X$, then
\begin{equation}\label{palaisdef}
\frac{1}{1+\varepsilon}||d\varphi^{-1}(\varphi(x_0))(v)||_{x_0}\le ||d\varphi^{-1}
(\varphi(x))(v)||_{x}\le (1+\varepsilon)||d\varphi^{-1}(\varphi(x_0))(v)||_{x_0}.
\end{equation}
In terms of equivalence of norms, the above inequalities yield to the fact that
the norms $||d\varphi^{-1}(\varphi(x))(\cdot)||_{x}$
and   $||d\varphi^{-1}(\varphi(x_0))(\cdot)||_{x_0}$ are $(1+\varepsilon)$-equivalent.

\end{enumerate}
\end{Def}

Let us remark that every Riemannian manifold is  a $C^\infty$ Finsler manifold in
 the sense of Palais (see  \cite{Palais}).
Throughout this work, we will  assume that $M$ is a (paracompact) connected ${C}^{1}$ Finsler manifold in the sense of Palais
modeled on a Banach space $X$. For simplicity we will  refer to them as
 ${C}^{1}$ Finsler manifolds.

Recall that the {\bf length} of a piecewise $C^{1}$ smooth
path $c:[a,b]\rightarrow M$ is defined as $\ell(c):=\int_{a}^b||c'(t)||_{c(t)}\,dt$.
 Besides, if $M$ is connected, then it is
connected by piecewise $C^1$ smooth paths, and the associated
{\bf Finsler metric} $d$ on $M$ is defined as
\begin{equation*}
d(p,q)=\inf\{\ell(c): \, c \text{ is a piecewise } C^1 \text{ smooth path connecting } p \text{ to } q\}.
\end{equation*}
We will say that a Finsler manifold $M$ is {\bf complete} if it is complete for the metric $d$. The following result  yields  the  local bi-Lipschitz behaviour of the
charts of a Finsler manifold.

\begin{Lem}\label{desigualdades:BiLipschitz}  \cite[Lemma 2.4.]{JSSG} {\bf (Bi-Lipschitz charts).}
Let us consider a  $C^1$ Finsler manifold $M$ modeled on a Banach space $X$ with a $C^1$
bump function and  $x_0\in M$.
Then, for every chart
$(U,\varphi)$  with $x_0\in U$ satisfying inequality \eqref{palaisdef},
there exists an open neighborhood
 $V\subset U$ of $x_0$ satisfying
\begin{equation}\label{B-L1}
(1+\varepsilon)^{-1}d(p,q)\le |||\varphi(p)-\varphi(q)|||\le (1+\varepsilon)
d(p,q), \quad \text{ for every } p,q\in V,
\end{equation}
where   $|||\cdot|||$ is the (equivalent) norm
$||d\varphi^{-1}(\varphi(x_0))(\cdot)||_{x_0}$ defined  on $X$.
\end{Lem}


The concepts of subdifferential and superdifferential have been extensively studied
for functions defined on $\mathbb R^n$, infinite dimensional Banach spaces and
Riemannian manifolds. The straightforward definition in the case of Finsler
manifolds is the following.

\begin{Def}
Let $M$ be a ${C}^1$ Finsler manifold modeled on a Banach space with a
${C}^1$ Lipschitz bump function and let $f\colon M\to (-\infty,+\infty]$
be a proper function. We define the set of {\bf subdifferentials} of $f$ at a point  $x\in dom(f)=\{y\in M:\, f(y)<\infty\}$
as
\begin{equation*}
D^-f(x)=\{\Delta\equiv dg(x)\,:\,g\colon M\to\mathbb{R}\mbox{ is } C^1
smooth\mbox{ and }f-g\mbox{ attains a local minimum at } x\}\subset T_xM^*,
\end{equation*}
and the set of {\bf superdifferentials} of $f$ at $x$ as
\begin{equation*}D^+f(x)=\{\Delta\equiv dg(x)\,:\,g\colon M\to\mathbb{R}\mbox{ is }
C^1 smooth\mbox{ and }f-g\mbox{ attains a local maximum at } x\}\subset
T_xM^*.
\end{equation*}
If $D^-f(x)\not= \emptyset$ ($D^+f(x)\not=\emptyset$), we say that $f$ is
subdifferentiable (superdifferentiable) at $x$.
\end{Def}\

\noindent Notice that if a function $f$ attains a local minimum at $x$, then  $0\in D^-f(x)$. Also notice that $D^-f(x)=-D^+(-f)(x)$. In addition, we can endow every
subdifferential or superdifferential $\Delta\in T_xM^*$ of $f$ at $x$  with
the dual norm
\begin{equation*}
\|\Delta\|_x^*=\sup\limits_{\xi\in {S}_{T_xM}} |\Delta(\xi)|, \qquad \text{ where }
 {S}_{T_xM}=\{\xi\in T_xM\,:\, \|\xi\|_x=1\}.
\end{equation*}

For simplicity we will write $\|\Delta\|_x$ for the dual norm $\|\Delta\|_x^*$.
Basic  properties related to subdifferentiability on Finsler manifolds can be
deduced in the same way as D. Azagra, J. Ferrera and F. L\'opez-Mesas did in
\cite[Section 4]{AFLM} for Riemannian manifolds. Since these properties can be
 deduced  without much difficulty by using the same techniques, we will omit some
 of the proofs.

\begin{Teo}\label{Teo43}{\bf (Characterizations of subdifferentiability).}
Let $M$ be a ${C}^1$ Finsler manifold modeled on a Banach space $X$ with a
${C}^1$ Lipschitz bump function. Consider a proper function
$f\colon M\to (-\infty,+\infty]$, a point $x\in M$ and a functional
$\Delta\in T_xM^*$. The following conditions are equivalent:
\newcounter{saveenum}
\begin{enumerate}
	\item $\Delta\in D^-f(x)$.
	\item There exists a function $g\colon M\to \mathbb{R}$ such that $g$ is
Fr\'echet differentiable at $x$, $f-g$ attains a local minimum at $x$ and
$\Delta=dg(x)$.
	\item For every chart $\varphi\colon U\subset M\to X$ with $x\in U$,
if we set $\tau:=\Delta\circ d\varphi^{-1}(\varphi(x))$, then
\begin{equation*}
\mathop{\lim\,\inf}\limits_{h\to 0}\dfrac{(f\circ\varphi^{-1})(\varphi(x)+h)-f(x)
-\tau(h)}{\|h\|}\geq 0,
\end{equation*}
i.e. $\tau \in D^-(f\circ \varphi^{-1})(\varphi(x))$.
\item There exists a chart $\varphi\colon U\subset M\to X$, with $x\in U$,
such that if $\tau=\Delta\circ d\varphi^{-1}(\varphi(x))$, then
\begin{equation*}
\mathop{\lim\,\inf}\limits_{h\to 0}\dfrac{(f\circ\varphi^{-1})(\varphi(x)+h)-f(x)
-\tau(h)}{\|h\|}\geq 0,
\end{equation*}
i.e. $\tau \in D^-(f\circ \varphi^{-1})(\varphi(x))$.

\end{enumerate}
Moreover, if $f$ is locally bounded below and $M$  admits $C^1$ smooth partitions of unity, we have the equivalent condition:
\begin{enumerate}
\item[(5)] There exists a $C^1$ smooth  function $g\colon M\to\mathbb{R}$,
 $f-g$ attains a global minimum at $x$ and $\Delta=dg(x)$.
\end{enumerate}
\end{Teo}
 \noindent Note that we can obtain an analogous result for the superdifferentiability
  of $f$. The proofs of (2) $\Longrightarrow$ (3) and (5) $\Longrightarrow$ (1) follow the lines of the Riemannian case \cite{AFLM}.
	The proof of (4) $\Longrightarrow$ (1) follows (via charts) from the  case of Banach spaces with a $C^1$ smooth bump \cite[Chapter 8]{DGZ}. 


Under the  assumptions of Theorem \ref{Teo43}, we get the following corollaries
related to the subdifferentiability and differentiability of $f$ at a point $x\in dom(f)$.

\begin{Cor} \label{subdiferencialporcartas} Let $M$ be a ${C}^1$ Finsler manifold modeled on a Banach space $X$
 with a ${C}^1$ Lipschitz bump function. Consider a proper function
 $f\colon M\to (-\infty,+\infty]$, a chart $\varphi\colon U\subset M\to X$
 and a point $x\in dom(f)\cap U$. Then,
\begin{align*}D^-f(x)&
=\left\{\tau\circ d\varphi(x)\,:\,\tau\in X^*,\,
 \mathop{\lim\inf}\limits_{h\to 0}\dfrac{(f\circ \varphi^{-1})
 (\varphi(x)+h)-f(x)-\tau(h)}{\|h\|}\geq 0\right\}\\
&=\{\tau \circ d\varphi(x)\,:\, \tau \in
 D^-(f\circ\varphi^{-1})(\varphi(x))\}.
\end{align*}
Moreover, $f$ is (Fr\'echet) differentiable at $x$ if and only if there exist an open subset $V$ in $M$ with
$x\in V$ and
${C}^1$ smooth functions $g,h\colon V\to\mathbb{R}$ such that
\begin{itemize}
	\item[(1)] $g(z)\leq f(z)\leq h(z)$ for all $z\in V$, and
	\item[(2)] $g(x)=f(x)=h(x)$ and
$dg(x)=dh(x)$.
\end{itemize}
\end{Cor}

The differentiability of $f$ is therefore characterized as follows.

\begin{Cor} {\bf (Criterion for differentiability).} Let $M$ be a ${C}^1$ Finsler manifold modeled on a Banach
space with a ${C}^1$ Lipschitz bump function. Consider a proper function
$f\colon M\to (-\infty,+\infty]$ and a point $x\in dom(f) $. Then, $f$ is (Fr\'echet) differentiable
at $x$ if and only if $f$ is subdifferentiable and superdifferentiable at $x$.
Moreover, if $f$ is (Fr\'echet) differentiable at $x$, then $df(x)$ is the only subdifferential
and superdifferential of $f$ at $x$.
\end{Cor}

As in the case of Banach spaces and Riemannian manifolds, the following relationship
 between the subdifferentiability and continuity holds.

\begin{Cor} {\bf (Continuity properties).} Let $M$ be a ${C}^1$ Finsler manifold modeled on a Banach
space with a ${C}^1$ Lipschitz bump function. Consider a proper function
$f\colon M\to (-\infty,+\infty]$ and a  point $x\in dom(f)$. If $f$ is subdifferentiable
(superdifferentiable) at $x$, then $f$ is lower semicontinuous (upper semicontinuous)
 at $x$.
\end{Cor}

The next  results are related to the subdifferentiability of the composition,
sum and product of functions defined on Finsler manifolds.

\begin{Pro} \label{chain1} {\bf (Chain rule).} Let $M,N$ be ${C}^1$
Finsler manifolds modeled on a Banach space with a ${C}^1$ Lipschitz
bump function. Let $g\colon M\to N$ and $f\colon N\to (-\infty,+\infty]$ be
 two functions such that $f$ is subdifferentiable at $g(x)$ and $g$ is Fr\'echet
  differentiable at $x$. Then $f\circ g$ is subdifferentiable at $x$ and
$$
\{\Delta\circ dg(x)\,:\,\Delta\in D^-f(g(x))\}\subset D^- (f\circ g)(x).
$$
\end{Pro}

\begin{Cor}\label{chain2}
Let $M,N$ be ${C}^1$ Finsler manifolds modeled on a Banach space with a
${C}^1$ Lipschitz bump function and assume that $\varphi\colon M\to N$
is a ${C}^1$ diffeomorphism. Then $f\colon M\to (-\infty,+\infty]$ is
subdifferentiable at $x$ if and only if $f\circ \varphi^{-1}$ is subdifferentiable
at $\varphi(x)$, and
\begin{equation*}
D^-f(x)=\{\Delta\circ d\varphi(x)\colon \Delta\in D^-(f\circ \varphi^{-1})(\varphi(x))\}.
\end{equation*}
\end{Cor}

\begin{Pro}\label{sume-product} Let $M$ be a ${C}^1$ Finsler manifold
modeled on a Banach space with a ${C}^1$ Lipschitz bump function and
consider the functions $f,g\colon M\to (-\infty,+\infty]$.
 Then the following statements hold:
\begin{enumerate}
	\item[(1)] {\bf (Sum rule).} $D^-f(x)+D^-g(x)\subset D^-(f+g)(x)$.
	\item[(2)] {\bf (Product rule).} If $f,\,g:M\rightarrow [0,\infty)$, then
 $f(x)D^-g(x)+g(x)D^-f(x)\subset D^-(fg)(x)$.
\end{enumerate}
\end{Pro}

Note that there are analogous statements of Propositions \ref{chain1} and
\ref{sume-product} and Corollary \ref{chain2}  for superdifferentials.

\begin{Pro} {\bf (Geometrical and topological properties of the subdifferencial).}
Let $M$ be a ${C}^1$ Finsler manifold modeled on a Banach space $X$ with a
${C}^1$ Lipschitz bump function. 
For every function $f:M\rightarrow (-\infty,\infty]$ and  $x\in dom(f)$,
 the sets $D^-f(x)$ and $D^+f(x)$ are closed and convex subsets of $T_xM^*$.
 Moreover, if $f$ is locally Lipschitz, then these sets are bounded.
\end{Pro}

The following results are fundamental  for the study of viscosity solutions of the
Hamilton-Jacobi equations on Finsler manifolds  given in the  next sections.

\begin{Pro}\label{fuzzy} {\bf (Fuzzy rule for the subdifferential of the sum).}
Let $M$ be a ${C}^1$ Finsler manifold modeled on a Banach space $X$ with a
${C}^1$ Lipschitz bump function. Let $f,g\colon M\to\mathbb{R}$ be two
 functions such that $f$ is lower semicontinuous and $g$ is locally uniformly
  continuous. Then, for every $x\in M$, every chart $(U,\varphi)$ with $x\in U$,
   every $\Delta\in D^-(f+g)(x)$ and $\varepsilon>0$, there exist $x_1,x_2\in U$,
   $\Delta_1\in D^-f(x_1)$, $\Delta_2\in D^-g(x_2)$ such that
\begin{enumerate}
	\item[(1)] $d(x_1,x)<\varepsilon$ and $d(x_2,x)<\varepsilon$,
	\item[(2)] $|f(x_1)-f(x)|<\varepsilon$ and $|g(x_2)-g(x)|<\varepsilon$,
	\item[(3)] $\|\Delta_1\circ d\varphi(x_1)^{-1}+\Delta_2\circ
d\varphi(x_2)^{-1}-\Delta\circ d\varphi(x)^{-1}\|<\varepsilon$.
\item[(4)] $d(x_1,x_2)\cdot \max \big\{ ||\Delta_1\circ d\varphi(x_1)^{-1}||,\, 
||\Delta_2\circ d\varphi(x_2)^{-1}||\big\}<\varepsilon$.
\end{enumerate}
\end{Pro}
\noindent The proof of the above fuzzy rule follows from the analogous results for Banach spaces \cite[Theorem 2.12]{BZ} and \cite[Theorem 4.2 in Section 4.2]{DG} 
applied to the functions $f\circ \varphi^{-1}$ and $g\circ \varphi^{-1}$ defined in a neighborhood of $\varphi(x)$, Lemma \ref{desigualdades:BiLipschitz} and Corollary \ref{subdiferencialporcartas}. Recall that $\varphi$ is locally bi-Lipschitz (Lemma 
\ref{desigualdades:BiLipschitz}) and then $g\circ \varphi^{-1}$ is locally uniformly continuous. It is worth noticing that the hypothesis given in the fuzzy rule for the subdifferential of the sum can be weakened by a
more technical  assumption (see \cite[Section 2]{BZ} and \cite[Section 4.2]{DG}).
Let us remark that up to our knowledge it is not known  whether the fuzzy rule holds for every pair of
lower semicontinuous functions with finite values $u,v:X\to \mathbb R$, where $X$ is a Banach spaces with a $C^1$ Lipschitz bump.

Recall that the smooth variational principle of Deville-Godefroy-Zizler for a Banach space $X$ with a $C^1$ Lipschitz bump function \cite{DGZ1,DGZ} provides
 the subdifferentiability of  a lower
semicontinuous function $f:X\to (-\infty,\infty]$ on a dense subset of $dom(f)=\{y\in X:\, f(y)<\infty\}$. There is a similar statement for Finsler manifolds.

\begin{Pro} {\bf (Density of the set of points of subdifferentiability).}
Let $M$ be a ${C}^1$ Finsler manifold modeled on a Banach space $X$ with a
${C}^1$ Lipschitz bump function. If $f\colon M\to (-\infty,+\infty]$ is
proper and lower semicontinuous, then the subset of points of $dom(f)$
where $f$ is subdifferentiable is dense in $dom(f)$.
\end{Pro}

\noindent Let us give an outline of the proof: Given a point $x\in dom(f)$ and a chart $(U,\varphi)$ with $x\in U$, we consider the lower semicontinuous function $L:X\to (-\infty,\infty] $
defined as $L=f\circ \varphi^{-1}$  in a closed neighborhood $C$ of $\varphi(x)$ ($C$ small enough such that $C\subset \varphi(U)$) and  $L=\infty$ in $X\setminus C$. The analogous result  on Banach spaces establishes that there is a sequence of subdifferentiable points of $L$ in $X$ with limit $\varphi(x)$. Thus, by Corollary \ref{subdiferencialporcartas}, there is a sequence of subdifferentiable points of $f$ in $U$
with limit  $x$.

Let us recall the well-known concepts of lower and upper semicontinuous envelopes of a function.

\begin{Def}
Let $M$ be a ${C}^1$ Finsler manifold modeled on a Banach space $X$. For a function $u:\Omega \to \Real$ defined on an open subset $\Omega \subset M$, the {\bf upper semicontinuous envelope} $u^*$ of $u$  is defined by
\begin{equation*}
u^*(x)=\inf\{v(x): \, v:\Omega \to \Real \text{ is continuous and } u\le v \text{ on } \Omega\} \qquad \text{for any $x\in \Omega$.}
\end{equation*}
\end{Def}
The {\bf lower semicontinuous envelope} $u_*$ is defined in a similar way.
 Recall that 
\begin{equation*}
u^*(x)=\lim_{r\to 0^+}\bigl(\sup_{y\in B(x,r)}u(y)\bigr) \  \  \  \text{ and } \  \ \
u_*(x)=\lim_{r\to 0^+} \bigl(\inf_{y\in B(x,r)}u(y)\bigr) \  \  \ \text{ for } x\in \Omega,
\end{equation*}
where $B(x,r)$ denotes the open  ball of center $x$ and
radius $r>0$ in the Finsler manifold $M$.
The following result of stability of superdifferentials is fundamental
  in the theory of viscosity solutions.

\begin{Pro}\label{conv} {\bf (Stability of the superdifferentials).}
Let $M$ be a ${C}^1$ Finsler manifold modeled on a Banach space $X$ with a
${C}^1$ Lipschitz bump function. Let $\Omega$ be an open subset of $M$. Let $\mathcal{F}$ be a locally uniformly bounded family of upper semicontinuous functions from $\Omega$ into $\mathbb R$ and $u=\sup\{v:\, v\in \mathcal{F}\}$ on $\Omega$. Then, for every $x\in \Omega$ and every $\Delta \in D^+u^*(x)$, there exist sequences  $\{v_n\}_{n\in \mathbb N}$ in $\mathcal{F}$ and $\{(x_n,\Delta_n)\}_{n\in \mathbb N}$ in $TM^*$ with $x_n \in \Omega$ and $\Delta_n \in D^+v_n(x_n)$ for every $n\in \mathbb N$, such that
\begin{itemize}
\item[(i)] $\lim_{n\to \infty} v_n(x_n)=u^*(x)$, and
\item[(ii)] $\lim_{n\to \infty} (x_n,\Delta_n) =(x,\Delta)$ in the cotangent bundle $TM^*$, i.e.
$\lim_{n\to \infty}d(x_n,x)=0$ and
 $\lim_{n \to \infty} ||\Delta_n\circ d\varphi(x_n)^{-1}-
\Delta\circ d\varphi(x)^{-1}||=0$ for every chart $(U, \varphi)$ on $M$ with $x\in U$. (Notice that, in general, we assume that 
$\Delta_n\circ d\varphi(x_n)^{-1}$ are defined only for $n\ge n_0$, where  $n_0$ depends on the chart $(U, \varphi)$).

\end{itemize}
\end{Pro}
\noindent Let us point out that the proof of Proposition \ref{conv} follows the lines
of the Riemannian case: for a fixed chart $(A,\psi)$ of $M$ with $x\in A\subset \Omega$,
we consider the functions $u\circ \psi^{-1} =\sup \{v\circ \psi^{-1} : \, v\in \mathcal F\}$
and $(u\circ \psi^{-1})^*=u^*\circ \psi^{-1}$, which are defined in the open neighborhood $\psi(A)$ of $\psi(x)\in X$.
Next, we apply the analogous result for Banach spaces to the function $u^*\circ \psi^{-1}$ \cite[Chapter VIII. Proposition 1.6]{DGZ} and Corollary \ref{subdiferencialporcartas} to obtain the assertions (i) and (ii)  for the chart $(A,\psi)$. Next,  it can be easily checked that, in fact, condition (ii) holds for every chart $(U,\varphi)$ with $x\in U$. 

\medskip

Now, let us give a local version of  Deville's mean value inequality, which will be
essential in order to prove the uniqueness of the eikonal equation on Finsler
manifolds. Recall that, for a Finsler manifold $M$, the open (closed) ball of center $x$ and
radius $r>0$ is denoted by $B(x,r)$ ($\overline{B}(x,r)$).

\begin{Teo} {\bf (Local Deville's mean value inequality for Finsler manifolds).}
\label{devillemean}
Let $M$ be a ${C}^1$ Finsler manifold modeled on a Banach space $X$ with a
${C}^1$ Lipschitz bump function and consider $p\in M$ and $\delta>0$.
Let $f\colon B(p,4\delta)\subset M\to\mathbb{R}$ be a lower
semicontinuous function satisfying  $\|\xi\|_x\le K$  for every $\xi\in D^-f(x)$ and $x\in B(p,4\delta)$.
Then $f$ is $K$-Lipschitz  on $B(p,\delta)$.
\end{Teo}

\begin{proof}
Let us fix $\varepsilon>0$ and consider for every pair of points $x,y\in B(p,\delta)$
 a continuous piecewise ${C}^1$ smooth path $\gamma\colon [0,T]\to M$ such that
 $\gamma(0)=x$, \,
  $\gamma(T)=y$ and $\ell(\gamma)<d(x,y)+\min\{\delta, \varepsilon\}$.
  Notice that in this case, $\ell(\gamma)<d(x,y)
  +\delta\le d(x,p)+d(p,y)+\delta<3\delta$.
  Thus, for every $z\in \gamma([0,T])$,
$d(p,z)\le d(p,x)+d(x,z)<\delta+\ell(\gamma)<\delta+3\delta=4\delta$ and
this yields $\gamma([0,T])\subset B(p,4\delta)$.

\medskip

Now, for all $z\in \gamma([0,T])$ we consider a chart
 $\varphi_z:U_z\rightarrow V_z$
such that
  \begin{enumerate}
  \item[(a)]$z\in U_z\subset B(p,4\delta)$,  $\varphi_z(U_z)=V_z$ and $V_z$ is an open convex subset of $X$,
  \item[(b)] $\varphi_z$ satisfies the Palais condition \eqref{palaisdef} for $1+\varepsilon$,  and
\item[(c)]$\varphi_z$  is $(1+\varepsilon)$-bi-Lipschitz in $U_z$ for the norm in $X$
denoted as
  \begin{equation*}
    \tresp{v}_{z}:=\|d\varphi_z^{-1}(\varphi_z(z))(v)\|_{z} \text{ for all } v\in X,
  \end{equation*}
    (see Definition \ref{defFinsler} and Lemma \ref{desigualdades:BiLipschitz}
     for more details).
    \end{enumerate}

 For every $t\in [0,T]$, we  select real  numbers
$0\le r_t< s_t\le T$ satisfying:  (1) $ r_0=0$  and $\gamma(0)=x\in
 \gamma([0, s_0])\subset U_{x}$, (2) $r_t<t<s_t$ and $\gamma(t)\in
 \gamma([ r_t, s_t])\subset U_{\gamma(t)}$ whenever $t\in (0,T)$,
 and (3) $ s_T= T$  and $\gamma(T)=y\in \gamma([r_T, T])\subset U_y$.

  By compactness of $[0,T]$, there exists a finite set of points
  $\{t_1,\ldots, t_n\}\subset [0,T]$ with $t_1=0$ and $t_n=T$ satisfying
\begin{equation}\label{dev1}
\gamma\left([0,T]\right)= \displaystyle\bigcup_{k=1}^n \gamma([ r_{t_k}, s_{t_k}]).
\end{equation}
Let us denote $z_k=\gamma(t_k)$, \,$\varphi_k:=\varphi_{z_k}$,\,  $U_k:=U_{z_k}$,
\,$V_k:=V_{z_k}$,
$r_k:=r_{t_k}$ and $s_k:=s_{t_k}$ for $k=1,\dots,n$. By reordering and splitting
 the intervals if needed, we may assume that $r_1=0$, $r_{k+1}=s_{k}$ for
 $k=1,\dots,n-1$ and $s_n=T$.

Consider the function $\Phi_k\colon V_k\subset X\to \mathbb{R}$ defined by
$\Phi_k=f\circ \varphi_k^{-1}$. By applying Corollary  \ref{chain2} we know that,
for all $a\in V_k$,
\begin{equation*}
D^-\Phi_k(a)=\{\Delta\circ d\varphi_k^{-1}(a)\,:\,
\Delta\in D^-f\left(\varphi_k^{-1}(a)\right)\}.
\end{equation*}
Now, for any $k=1,\dots,n$ we consider in $X$ the norm
\begin{equation*}
\tresp{v}_{k}:=\tresp{v}_{z_k}=\|d\varphi_k^{-1}(\varphi_k(z_k))(v)\|_{z_k} \quad
\text{ for all } v\in X.
 \end{equation*}
If  $z\in U_k$, for a continuous linear  operator $T: (T_z M, \|\cdot\|_z) \rightarrow
(X, \tresp{\cdot}_{k})$  we set the norm
\begin{equation*}
|||T|||_{z,k}:=\sup \{\tresp{T(v)}_{k}: \, \|v\|_{z}\le 1\}.
\end{equation*}
Moreover, if $T$ is an isomorphism we denote
\begin{equation*}
|||T^{-1}|||_{k,z}=\sup\{\|T^{-1}(v)\|_{z}: \, \tresp{v}_{k}\le 1\}.
\end{equation*}

From the Palais condition \eqref{palaisdef}, we obtain for all $z\in U_k$ and
 $v\in T_z M$,
\begin{equation*}
\tresp{d\varphi_k(z)(v)}_{k}=\| d\varphi_k^{-1}(\varphi_k(z_k) )(d\varphi_k(z)(v))
 \|_{z_k}
\le (1+\varepsilon) \| d\varphi_k^{-1}(\varphi_k(z) )(d\varphi_k(z)(v)) \|_{z}
=(1+\varepsilon)\|v\|_{z}
\end{equation*}
and
\begin{equation*}
\tresp{d\varphi_k(z)(v)}_{k}=\| d\varphi_k^{-1}(\varphi_k(z_k))(d\varphi_k(z)(v))
 \|_{z_k}
\ge (1+\varepsilon)^{-1} \| d\varphi_k^{-1}(\varphi_k(z) )(d\varphi_k(z)(v))
 \|_{z}=(1+\varepsilon)^{-1}\|v\|_{z}.
\end{equation*}
Therefore, for all $z\in U_k$,
\begin{equation*}
(1+\varepsilon)^{-1}\le |||d\varphi_k(z)|||_{z,k} \le (1+\varepsilon) \quad
\text{ and thus } \quad
(1+\varepsilon)^{-1}\le||| d\varphi_k^{-1}(\varphi_k(z)) |||_{k,z}\le
(1+\varepsilon).
\end{equation*}

Now,  for all $a\in V_k$,  with $z=\varphi_k^{-1}(a)\in U_k$ and
$\Delta\in D^-f(\varphi_k^{-1}(a))$,  we have
\begin{equation*}
\tresp{\Delta\circ d\varphi_k^{-1}(a)}_{k}^* \leq \|\Delta\|_{z}^*\,
|||d\varphi_k^{-1}(a)|||_{k,z}\leq K(1+\varepsilon).
 \end{equation*}
Therefore,   $\tresp{\Lambda}_k^*\leq K(1+\varepsilon)$ for all
$\Lambda\in D^-\Phi_k(a)$ and $a\in V_k$.

Let us define $x_k=\gamma(r_k)$ for all
 $k=1,\dots, n$ and $x_{n+1}=\gamma(s_n)$. The function $\Phi_k: V_k
 \to \mathbb R$ is lower semicontinuous at every point of the open convex set
 $V_k\subset X$. Let us apply Deville's mean value inequality for Banach spaces
  (\cite{D2}; see also \cite{D, DG}) to the function $\Phi_k$ defined in the open convex subset $V_k$ of  the Banach space
  $(X,\tresp{\cdot}_{k})$ to obtain that $\Phi_k$ is $K(1+\varepsilon)$-Lipschitz with respect to the norm $\tresp{\cdot}_{k}$ and then,
\begin{equation*}
|f(x_{k+1})-f(x_k)|=|\Phi_k(\varphi_k(x_{k+1}))-\Phi_k(\varphi_k(x_k))|
\leq K(1+\varepsilon) \tresp{\varphi_k(x_{k+1})-\varphi_k(x_{k})}_{k}
 \end{equation*}
for all $k=1,\ldots, n$. In addition, since $\varphi_k$ is 
$(1+\varepsilon)$-bi-Lipschitz in $U_k$ with the norm $\tresp{\cdot}_{k}$ (Lemma \ref{desigualdades:BiLipschitz}),
 we obtain
\begin{equation*}
|f(x_{k+1})-f(x_{k})|\leq K(1+\varepsilon)^2d(x_{k+1},x_{k})
\end{equation*}
for all $k=1,\ldots, n$. Consequently,
\begin{equation*}
|f(x)-f(y)|\leq \sum\limits_{k=1}^n|f(x_{k+1})-f(x_k)|\leq K(1+\varepsilon)^2
\sum\limits_{k=1}^n d(x_{k+1},x_k)\leq K(1+\varepsilon)^2 \ell(\gamma)\leq
K(1+\varepsilon)^2( d(x,y)+\varepsilon).
\end{equation*}
By letting $\varepsilon\to 0$, we get the inequality $$|f(x)-f(y)|\leq
Kd(x,y).$$ Finally, since $x,y\in B(p,\delta)$ are arbitrary, $f$ is
$K$-Lipschitz in $B(p,\delta)$.\end{proof}


By applying the same techniques of the above theorem,
 we can prove a global mean value inequality for Finsler manifolds.

\begin{Teo} {\bf (Deville's mean value inequality for Finsler manifolds).}
Let $M$ be a ${C}^1$ Finsler manifold modeled on a Banach space with a
${C}^1$ Lipschitz bump function. Let $f\colon M\to\mathbb{R}$ be a
lower semicontinuous function such that
$\|\xi\|_x\le K$  for every $\xi\in D^-f(x)$ and $x\in M$.
Then,  $f$ is $K$-Lipschitz.
\end{Teo}

\section{The eikonal equation on Banach-Finsler manifolds}

Let $M$ be a {\bf complete} ${C}^1$ Finsler manifold modeled on a Banach space with a
${C}^1$ Lipschitz bump function and assume that $\Omega\subset M$ is a
non-empty bounded open subset with $\partial \Omega\not= \emptyset$. Let us consider the {\bf eikonal equation},
\begin{equation}\label{EEq}\tag{EEq}
\left\{\begin{array}{ll}\|du(x)\|_x=1,&\mbox{ for all }x\in\Omega,\\
u(x)=0,&\mbox{ for all }x\in\partial\Omega\end{array}\right.
\end{equation}
which is a well-known Hamilton-Jacobi equation. Our purpose throughout this
section is to prove that this equation has a unique viscosity solution.
Let us first see that \eqref{EEq} does not have a classical solution.

\begin{Pro} \label{noclassical}
\eqref{EEq} does not have a classical solution.
\end{Pro}
\begin{proof}
Assume that there exists a classical solution  $u: \overline{\Omega}\to
\mathbb{R}$ of \eqref{EEq}, i.e. $u$ is continuous in $\overline{\Omega}$,
(Fr\'echet) differentiable in $\Omega$ and $||du(x)||_x=1 $ for all $x\in \Omega$.
 We extend $u$ to $M$ as $u(z)=0$ for $z\in \Omega^c$.
 By applying Theorem \ref{devillemean}, let us check that  $u$ is $1$-Lipschitz:

\noindent (i) For $x,y \in \Omega$, consider a piecewise $C^1$  smooth path
$\gamma:[a,b]\rightarrow M$ with $\gamma(a)=x$
and $\gamma(b)=y$. 
\begin{enumerate}
\item[(a)] If  $\gamma([a,b])\subset \Omega$, we apply
the compactness of $\gamma([a,b])$ to find
a finite number of balls $\{B(x_i,4\delta_i):\, i=1\dots,m\}$
such that
\begin{equation*}
\bigcup_{i=1}^m B(x_i,4\delta_i)\subset \Omega \ \text{ and } \
\gamma([a,b])\subset \bigcup_{i=1}^m B(x_i,\delta_i).
\end{equation*}
We may assume that there are auxiliary points $a=t_1<t_2<\cdots t_{n+1}=b$
 such that for every $k\in \{ 1,\dots, n\}$ there is  $i\in \{1,\dots,m\}$
 satisfying $\gamma([t_{k},t_{k+1}])\subset B(x_{i},\delta_{i})$.
 By applying Theorem \ref{devillemean}, we deduce that
\begin{equation*} \label{uLip}
|u(x)-u(y)|\le \sum_{k=1}^{n}|u(\gamma(t_{k+1}))-u(\gamma(t_k))|\le
 \sum_{k=1}^{n}d(\gamma(t_{k+1}),\gamma(t_k))\le  \ell (\gamma).
\end{equation*}
\item[(b)] If there is $b'\in (a,b]$ such that $\gamma([a,b'))\subset \Omega$ and $\gamma(b')\in \partial \Omega$, by taking the restrictions $\gamma|_{[a,t]}$ with $t<b'$ and the limit $t\to {b'}$, we obtain 
\begin{equation*}
|u(\gamma(b'))-u(\gamma(a))|\le \ell(\gamma|_{[a,b']}).
\end{equation*}

\noindent Thus, if $\gamma([a,b])\not\subset \Omega$, consider the points
$a<a'\le b'<b$ such that $\gamma([a,a'))\subset \Omega$,
$\gamma(a')\in \partial \Omega$,  $\gamma((b',b])\subset \Omega$ and
$\gamma(b')\in \partial \Omega$.
Then, by the preceding observation
\begin{equation*}|u(\gamma(a))|=|u(\gamma(a))-u(\gamma(a'))|\le
 d(\gamma(a),\gamma(a'))\le \ell(\gamma|_{[a,a']})
 \end{equation*}
 and
\begin{equation*} |u(\gamma(b))|=|u(\gamma(b'))-u(\gamma(b))|\le
 d(\gamma(b'),\gamma(b))\le
\ell(\gamma|_{[b',b]}).
 \end{equation*}
 Therefore, $|u(\gamma(a))-u(\gamma(b))|\le \ell(\gamma|_{[a,a']})
 +\ell(\gamma|_{[b',b]})\le \ell (\gamma)$.

\end{enumerate}

Thus, by taking the infimum of the lengths of all piecewise $C^1$ smooth
 paths $\gamma$ connecting $x$ and $y$, we obtain
$|u(x)-u(y)|\le d(x,y)$.

\noindent (ii) For $x\in \Omega$, $y\in \Omega^c$ and any   piecewise $C^1$ smooth path
$\gamma:[a,b]\rightarrow M$ with $\gamma(a)=x$
and $\gamma(b)=y$, consider the  point $a'$ in the segment $[a,b]$ such that
$\gamma([a,a'))\subset \Omega$
and $\gamma(a')\in \partial \Omega$. By the preceding  cases,
\begin{equation*}
|u(x)-u(y)|=|u(\gamma(a))|=|u(\gamma(a))-u(\gamma(a'))|\le d(\gamma(a),\gamma(a'))
 \le \ell(\gamma|_{[a,a']})\le  \ell(\gamma).
\end{equation*}
Again, by taking the infimum of the lengths of all piecewise $C^1$ smooth paths $\gamma$
connecting $x$ and $y$, we obtain
$|u(x)-u(y)|\le d(x,y)$.

\noindent (iii) For $x,y\in \Omega^c$ the inequality is clear.

\medskip

Since $\Omega$ is a bounded subset and $u$ is Lipschitz, $u$ is bounded on
$\overline{\Omega}$.
Notice that $-u$ is also a  classical solution of \eqref{EEq}, and thus we may assume that
$s=\sup\{ u(x):\, x\in\Omega\}>0$. Let us fix $0<\varepsilon<\min\{1,s\}$ and apply
 the Ekeland variational principle to $u:M\to \mathbb R$ (recall that we are assuming the completeness of $M$)  to find a point $\overline{x}\in
 M$ such that $s\leq u(\overline{x})+\varepsilon$ and
  $u(x)\leq u(\overline{x})+\varepsilon d(x,\overline{x})$ for all
  $x\in M$. Necessarily, $\overline{x}\in \Omega$
  (otherwise, $s\leq \varepsilon$, which is a contradiction).

Now, for each $v\in T_{\overline{x}}M$ with $||v||_{\overline{x}}=1$,
let us consider a piecewise ${C}^1$ smooth path $\gamma_v\colon [0,T]\to M$,
parametrized by the arc length, such that $\gamma_v(0)=\overline{x}$ and
$\gamma_v'(0)=v$.
Since $d(\gamma_v(0),\gamma_v(t))\leq \ell\left(\gamma_v|_{[0,t]}\right)=t$
for all $t\in [0,T]$, we have that $u(\gamma_v(t))-u(\gamma_v(0))\leq \varepsilon
d(\gamma_v(t),\gamma_v(0))\leq \varepsilon t$ for all $t\in [0,T]$. Therefore,
\begin{equation*}
du(\overline{x})(v)=\lim\limits_{t\to 0^+}\dfrac{u(\gamma_v(t))-u(\gamma_v(0))}{t}
\leq \varepsilon
\end{equation*}
and consequently, $\|du(\overline{x})\|_{\overline{x}}\leq\varepsilon<1$.
 This contradicts that $u$ is a classical solution of \eqref{EEq}.
\end{proof}

Let us consider the more general Hamilton-Jacobi equation
\begin{equation}\label{EEq2}\tag{EEq2}
\left\{\begin{array}{ll}\|du(x)\|_x=1,&\mbox{ for all }x\in\Omega,\\
u(x)=h(x),&\mbox{ for all }x\in\partial\Omega\end{array}\right.
\end{equation}
where $\Omega\subset M$ is a
non-empty bounded open subset with $\partial \Omega\not= \emptyset$ and $h:\partial \Omega \to \mathbb R$ is $1$-Lipschitz.
The definition of viscosity solution of the Hamilton-Jacobi equation \eqref{EEq2} on a Finsler
manifold is the following.
\begin{Def}
Let us consider a function $u\colon \overline{\Omega}\to\mathbb{R}$.

\begin{enumerate}
\item[(1)] $u$ is a \textbf{viscosity subsolution of
 \eqref{EEq2}} whenever $u$ is upper semicontinuous, $\|\Lambda\|_x\le 1$
 for all $\Lambda\in D^+u(x)$ with $x\in \Omega$ and $u\le h$ on $\partial\Omega$.
\item[(2)] $u$ is a \textbf{viscosity supersolution of \eqref{EEq2}}
 whenever $u$ is lower semicontinuous, $\|\Delta\|_x\ge 1$, for all $\Delta\in
 D^-u(x)$ with $x\in \Omega$ and $u\ge h$ on $\partial\Omega$.
\item[(3)]  $u$ is a \textbf{viscosity solution of \eqref{EEq2}}
 if $u$ is simultaneously a viscosity subsolution and a viscosity supersolution
 of \eqref{EEq2}, i. e.
$u$ is a continuous function and verifies
\begin{itemize}
	\item[(i)] $\|\Delta\|_x\geq 1$ for all $\Delta\in D^-u(x)$ with $x\in \Omega$,
	\item[(ii)] $\|\Lambda\|_x\leq 1$ for all $\Lambda\in D^+u(x)$ with $x\in \Omega$, and
	\item[(iii)] $u(x)=h(x)$ for all $x\in \partial \Omega$.
\end{itemize}
\end{enumerate}
\end{Def}

The next theorem shows that the equation \eqref{EEq2} has a unique viscosity solution.

\begin{Teo}\label{uniqeik}
 The function $u\colon \overline{\Omega}\to\mathbb{R}$,
 defined by $u(x)=\inf\{h(y)+d(y,x):\, y\in\partial\Omega\}$ is the
 unique viscosity solution of \eqref{EEq2}.
\end{Teo}
\begin{proof}
Since $h$ is $1$-Lipschitz, $h(x)-h(y)\le d(x,y)$  for every $x,y\in \partial \Omega$,  and then $h(x)\le h(y)+ d(x,y)$. By taking
 the infimum over all $y\in \partial\Omega$ we have $h(x)\le \inf\{h(y)+ d(x,y):\, y\in \partial\Omega\} \le h(x)+d(x,x) = h(x)$.
Thus $u(x)=h(x)$ for $x\in \partial\Omega$.

Now, let us check the conditions 
 (i) and (ii) given in the definition of viscosity solution.
We can consider  $u$ defined in $M$ with the same expression $u(x)=\inf\{h(y)+d(y,x):\, y\in\partial\Omega\}$ for $x\in \Omega^c$.
Let us first check (i). Consider  $\Delta\in D^-u(x)$ with $x\in\Omega$
and fix $\varepsilon>0$. Then, for every $\delta>0$, there exists $x_\delta\in \partial \Omega$
 such that
\begin{equation*}\label{e1}h(x_\delta)+d(x_\delta,x)\leq u(x)
+\dfrac{\delta\varepsilon}{2}.
\end{equation*}
Let us point out that, in the Finsler distance, it is possible to approximate $d(z,w)$ for $z,w\in M$ by the length of a $C^1$ smooth path connecting $z$ and $w$ and parametrized by the arc length. Let us give an outline of this fact: For a piecewise $C^1$ smooth path
$\rho:[a,b]\to M$ connecting $z$ and $w$ whose length approximates $d(z,w)$ and for  any $r>0$, we can find a finite collection of points 
$a=t_1<\cdots<t_{n+1}=b$ and a finite family of $(1+r)$-bi-Lipschitz charts 
$\{(A_i,\psi_i)\}_{i=1}^n$ given by  Lemma \ref{desigualdades:BiLipschitz}
such that $\rho([t_i,t_{i+1}])\subset A_i$ and $\psi_i(A_i)$ is open and convex in $X$. Now, we proceed in 
$X$ to construct a $C^1$ smooth path $\sigma_i:[t_i,t_{i+1}]\to X$  connecting $\psi_i(\rho(t_i))$ and $\psi_i(\rho(t_{i+1}))$ such that the length of $\sigma_i$ for the norm $|||u|||_{\rho(t_i)}:=||d\psi_i^{-1}(\psi_i(\rho(t_i)))(u)||_{\rho(t_i)}$ approximates
$|||\psi_i(\rho(t_i))-\psi_i(\rho(t_{i+1}))|||_{\rho(t_i)}$, $\sigma_i([t_i,t_{i+1}])\subset \psi_i(A_i)$,  $\sigma_i'(t)\not=0$ for every $t\in [t_i,t_{i+1}]$ and $i=1,\dots,n$ and $(\psi_i^{-1}\circ\sigma_i)'(t_{i+1})=(\psi_{i+1}^{-1}\circ\sigma_{i+1})'(t_{i+1})$
for every $i=1,\dots,n-1$. In this way, the length of $\psi_i^{-1}\circ\sigma_i:[t_i,
t_{i+1}]\to M$ approximates $d(\rho(t_i),\rho(t_{i+1}))$. Now the path given by the union 
$\sigma:=\cup_{i=1}^n (\psi_i^{-1}\circ\sigma_i):[a,b]\to M$ is a $C^1$ smooth path connecting $z$ and $w$, $\sigma'(t)\not=0$ for every $t\in [a,b]$ and $\ell(\sigma)$
approximates the distance $d(z,w)$ for $r>0$ small enough. Now, we can reparametrize $\sigma$ by the arc length to obtain the required $C^1$ smooth path.

Thus, we may assume that there are $C^1$ smooth paths  $\gamma_\delta\colon [0,T_\delta]\to M$
 parametrized by the arc length with
$||\gamma_{\delta}'(t)||_{\gamma_{\delta}(t)}=1$ for all $t\in [0,T_\delta]$ connecting $x$ and
$x_\delta$ (i.e., $\gamma_\delta(0)=x$ and $\gamma_\delta(T_\delta)=x_\delta$)
 and verifying
 \begin{equation*}
 \label{length}
 \ell(\gamma_\delta)=T_\delta \leq d(x,x_\delta)
 +\dfrac{\delta\varepsilon}{2}.
 \end{equation*}
Notice that $\delta<T_\delta$ whenever $\delta< d(x,\partial\Omega)$. So, let us define $z_\delta:=\gamma_\delta(\delta)\in M$ for $\delta < d(x,\partial\Omega)$. Then $d(x,z_\delta)\le \ell(\gamma_\delta|_{[0,\delta]})= \delta$ and thus
$\lim_{\delta \to 0 } d(x,z_{\delta})=0$.
 Since $\Delta\in D^-u(x)$, there exists a $C^1$ smooth function $g:M \to\mathbb{R}$
 such that $u-g$ attains a local minimum at $x$ and $\Delta = dg(x)$. Therefore $u(x)-g(x)\leq u(y)-g(y)$,
 for all $y$ in a neighbourhood of $x$. Thus, $u(x)-g(x)\leq u(z_\delta)-g(z_\delta)$
  for $\delta>0$ small enough. This yields
\begin{align*}
g(z_\delta)-g(x)&\leq u(z_\delta)-u(x)=\inf\{h(y)+d(y,z_\delta):\, y\in \partial \Omega\}-\inf\{h(y)+d(y,x):\, y\in \partial \Omega\} \\
&\le \inf\{h(y)+d(y,z_\delta):\, y\in \partial \Omega\}-h(x_\delta)-d(x_\delta,x)+\frac{\delta\varepsilon}{2}
\\&\leq h(x_\delta)+d(x_\delta,z_\delta)-h(x_\delta)-d(x_\delta,x)+\frac{\delta\varepsilon}{2}
\\&\leq \ell\left(\gamma_\delta\big{|}_{[\delta,T_\delta]}\right)
-\ell(\gamma_\delta)+\delta\varepsilon = -
 \ell\left(\gamma_\delta\big{|}_{[0,\delta]}\right)+\delta\varepsilon=
 -\delta + \delta\varepsilon=
\delta(\varepsilon-1).
\end{align*}
This implies
\begin{equation*}
\dfrac{g(z_\delta)-g(x)}{\delta}=\dfrac{g\circ \gamma_\delta(\delta)
-g\circ\gamma_\delta(0)}{\delta}\leq \varepsilon-1.
\end{equation*}
Since $g\circ \gamma_\delta$ is $C^1$ smooth, by the mean value theorem  there is
$\tau_\delta\in [0,\delta]$ such that
\begin{equation*}
\frac{\left|g\circ \gamma_\delta(\delta)-g\circ\gamma_\delta(0)\right|}{\delta}=
\left|(g\circ \gamma_\delta)'(\tau_\delta) \right|
\le ||dg(\gamma_\delta(\tau_\delta))||_{\gamma_\delta(\tau_\delta)}
 ||\gamma_\delta'(\tau_\delta)||_{\gamma_\delta(\tau_\delta)}
=||dg(\gamma_\delta(\tau_\delta))||_{\gamma_\delta(\tau_\delta)}.
\end{equation*}
Clearly, $d(x,\gamma_\delta(\tau_\delta))\le \ell(\gamma_\delta|_{[0,\tau_\delta]})
=\tau_\delta \le \delta$ and thus $\lim_{\delta \to 0 } d(x,\gamma_\delta(\tau_\delta))=0$.
Since the function $z \to ||dg(z)||_{z}$ is continuous, $\lim_{\delta\to 0}
 ||dg(\gamma_\delta(\tau_\delta))||_{\gamma_\delta(\tau_\delta)}=||dg(x)||_{x}$.
Thus $||dg(x)||_{x}\ge 1-\varepsilon.$
This inequality holds for every $\varepsilon>0$, and consequently
$\|\Delta\|_x\geq 1.$

\medskip

Now, let us show  (ii). Take $\Lambda\in D^+u(x)$, $x\in \Omega$.
There exists a $C^1$ smooth function $g:M \to\mathbb{R}$ such that $u-g$ attains a
local maximum at $x$ and $\Lambda=dg(x) $.
Therefore $u(y)-g(y)\leq u(x)-g(x)$, for all $y$ in a neighborhood of $x$.
For each $v\in T_xM$ with $\|v\|_x=1$, choose a (piecewise) $C^1$ smooth path parametrized
by the arc length $\gamma_v\colon [0,T]\to M$ such that $\gamma_v(0)=x$ and
$\gamma_v'(0)=v$. Then
\begin{equation*}
d(\gamma_v(0),\gamma_v(t))\leq \ell\left(\gamma_v|_{[0,t]}\right)
=t,\mbox{ for all }t\in [0,T].
\end{equation*}
It can be easily checked that  $u(x)=\inf\{h(y)+d(y,x):\, y\in \partial \Omega\}$ is 1-Lipschitz in $M$,
and thus for $t>0$ small enough
\begin{equation*}
g(\gamma_v(t))-g(\gamma_v(0))\geq u(\gamma_v(t))
-u(\gamma_v(0))\geq -d(\gamma_v(t),\gamma_v(0))\geq -t,
\end{equation*}
and
\begin{equation*}
dg(x)(v)=\lim\limits_{t\to 0^+}\dfrac{g(\gamma_v(t))-g(\gamma_v(0))}{t}\geq -1.
\end{equation*}
Therefore $dg(x)(-v)\leq 1$ and we can conclude that
$\|\Lambda\|_x=\|dg(x)\|_x\leq 1$.\\

\begin{Rem}
Following the above argument, we can prove the next statement: Let $M$ be a ${C}^1$ Finsler manifold modeled on a Banach space with a
${C}^1$ Lipschitz bump function. Assume that $\Omega\subset M$ is a non-empty
open subset and consider a function $f\colon \Omega\to \mathbb{R}$. If $f$ is
pointwise $K$-Lipschitz at  $x\in \Omega$, that is $|f(x)-f(y)|\le K d(x,y)$ for all $y $ in a neighborhood of $x$, then $||\Delta||_x\le K$ for every $\Delta \in D^+f(x)\cup D^-f(x)$.
\end{Rem}

Finally, we will check the uniqueness of the viscosity solution.
Suppose that there exist
two viscosity solutions $u,v\colon \overline{\Omega}\to\mathbb{R}$.
In particular, their superdifferentials at every point
 $x\in \Omega$ are $\|\cdot\|_x$-bounded above by $1$. Thus, $-u$ and $-v$ have
  subdifferentials $\|\cdot\|_x$-bounded above by $1$ in $\Omega$.
By applying  Theorem \ref{devillemean} we can deduce that $-u$ and $-v$ are locally
 $1$-Lipschitz in $\Omega$. We consider $u$ and $v$ defined  in
 $M\setminus \Omega$ as $u(x)=v(x)=\inf \{h(y)+d(y,x):\, y\in \partial \Omega\}$.
 Thus, $u,v:M\rightarrow \mathbb R$ are continuous, locally
 $1$-Lipschitz in $\Omega$, and $1$-Lipschitz in $M\setminus \Omega$.
 Following an analogous proof to the one given in Proposition \ref{noclassical},
 it can be  deduced that $u$ and $v$ are $1$-Lipschitz in $M$.

 Since $\Omega$ is bounded and $u$ and $v$ are $1$-Lipschitz, we know that $u$ and $v$
 are bounded in $\overline{\Omega}$.
 In fact, we may assume that the boundary data $h$ is non-negative in $\partial \Omega$.
  Otherwise, we consider $S>0$ large enough so that $\widetilde{h}=h+S$ is non-negative in
  $\partial \Omega$ and the Hamilton-Jacobi equation
  \begin{equation} \label{+M}
\begin{cases}\|d\widetilde{u}(x)\|_x=1,&\text{ for all }x\in\Omega,\\
\widetilde{u}(x)=\widetilde{h}(x),&\text{ for all }x\in\partial\Omega.\end{cases}
\end{equation}
Notice that a function $\widetilde{u}$ is a viscosity solution of \eqref{+M} if and only if $u=\widetilde{u}-S$ is
 a viscosity solution of \eqref{EEq2}.

Now, if we prove
 that $\theta u(x)\leq v(x)$ for all $x\in\overline{\Omega}$ and all
 $\theta\in (0,1)$, then we will have $u\leq v$. Analogously, it can be proved
  $v\leq u$, and thus $u=v$.

 Assume, by contradiction,
that $\sup\limits_{\overline{\Omega}} (\theta u-v)>0$ for some $\theta\in (0,1)$.
 We know that $\theta u-v$ is continuous and bounded.
 Hence, by applying the Ekeland variational principle to the function $\theta u-v:\overline{\Omega}
\to \mathbb R$ for
 $0<\varepsilon<\sup\limits_{\overline{\Omega}} (\theta u-v)$, we can find
 $\overline{x}\in \overline{\Omega}$ such that
\begin{equation*}
\sup\limits_{\overline{\Omega}} (\theta u-v)<(\theta u-v)(\overline{x})
+\varepsilon
 \end{equation*}
 and
 \begin{equation*}
 (\theta u-v)(x)\leq (\theta u -v)(\overline{x})+\varepsilon d(x,\overline{x}),
 \mbox{ for all }x\in \overline{\Omega}.
 \end{equation*}
Necessarily, $\overline{x}\in\Omega$, otherwise, $\sup\limits_{\overline{\Omega}}(\theta u-v)<
(\theta u - v)(\overline{x})+\varepsilon= (\theta-1) h(\overline{x})+\varepsilon\le \varepsilon$, which is a
contradiction. Since $(\theta u-v)(\cdot)-\varepsilon d(\cdot,\overline{x})$
 attains a local maximum at $\overline{x}$, we have $0\in D^+\left(\theta u(\cdot)
 -v(\cdot)-\varepsilon d(\cdot,\overline{x})\right)(\overline{x})$, which yields
 $0\in D^-(\varepsilon d(\cdot,\overline{x})+v(\cdot)-\theta u(\cdot))(\overline{x})$.


Let $(U,\varphi)$ be a chart with $\overline{x}\in U\subset \Omega$ satisfying the Palais condition
 for $1+\varepsilon$.
  Let us consider in
$X$ the norm $\tresp{v}_{\overline{x}}=\|d\varphi^{-1}(\varphi(\overline{x}))(v)
\|_{\overline{x}}$ for all $v\in X$. For a continuous linear  operator
 $T: (T_x M, \|\cdot\|_x) \rightarrow (X, \tresp{\cdot}_{\overline{x}})$,
  where $x\in U$, we consider  the norm
\begin{equation*}
 |||T|||_{x,\overline{x}}=\sup \{\tresp{T(v)}_{\overline{x}}: \, \|v\|_{x}\le 1\}.
 \end{equation*}
 Moreover, if $T$ is an isomorphism we consider the norm
 \begin{equation*}
|||T^{-1}|||_{\overline{x},x}=\sup\{\|T^{-1}(v)\|_{x}: \, \tresp{v}_{\overline{x}}\le 1\}.
  \end{equation*}
From the Palais condition, we obtain for all $x\in U$ and $v\in T_x M$,
\begin{equation*}
\tresp{d\varphi(x)(v)}_{\overline{x}}=\| d\varphi^{-1}(\varphi({\overline{x})} )(d\varphi(x)(v))
\|_{\overline{x}}
\le (1+\varepsilon) \| d\varphi^{-1}(\varphi(x) )(d\varphi(x)(v)) \|_{x}=(1+\varepsilon)\|v\|_x
\end{equation*}
and
\begin{equation*}
\tresp{d\varphi(x)(v)}_{\overline{x}}=\| d\varphi^{-1}(\varphi({\overline{x})} )(d\varphi(x)(v))
 \|_{\overline{x}}
\ge (1+\varepsilon)^{-1} \| d\varphi^{-1}(\varphi(x) )(d\varphi(x)(v)) \|_{x}=(1+\varepsilon)^{-1}\|v\|_x.
\end{equation*}
Therefore, for all $x\in U$,
\begin{equation*} (1+\varepsilon)^{-1}\le |||d\varphi(x)|||_{x,\overline{x}} \le (1+\varepsilon)
\text{ and thus }
(1+\varepsilon)^{-1}\le||| d\varphi^{-1}(\varphi(x)) |||_{\overline{x},x}\le (1+\varepsilon) .
\end{equation*}
For a continuous linear functional $L:(X, \tresp{\cdot}_{\overline{x}})\rightarrow \mathbb R$,
 we will consider the  norm
\begin{equation*}
\tresp{L}_{\overline{x}}=\sup\{|L(v)|: \, \tresp{v}_{\overline{x}}\le 1\}.
\end{equation*}
By applying   Proposition \ref{fuzzy} (the fuzzy rule for the subdifferential of the sum) to  the function $\varepsilon d(\cdot,\overline{x})+v(\cdot)
-\theta u(\cdot)$, we find points $x_1,x_2,x_3\in U\subset \Omega$, functionals $\Delta_1\in D^-(-\theta u)(x_1)$,
$\Delta_2\in D^- v(x_2)$, $\Delta_3\in D^-\left(\varepsilon d(\cdot,\overline{x})\right)(x_3)$ such that
\begin{equation}
\label{jota3}\tresp{\Delta_1\circ d\varphi(x_1)^{-1}+\Delta_2\circ d\varphi(x_2)^{-1}+\Delta_3\circ
d\varphi(x_3)^{-1}-0\circ d\varphi(\overline{x})^{-1}}_{\overline{x}}\leq\varepsilon.
\end{equation}
For convenience, we define
\begin{itemize}
	\item $\Lambda_1:=-\dfrac{1}{\theta}\Delta_1\in -\dfrac{1}{\theta}D^-(-\theta u)(x_1)=D^+u(x_1)$,
	\item $\Lambda_2:=\Delta_2\in D^- v(x_2)$,
	\item $\Lambda_3:=\dfrac{1}{\varepsilon}\Delta_3\in \dfrac{1}{\varepsilon}D^-\left(\varepsilon
d(\cdot,\overline{x})\right)(x_3)=D^-\left(d(\cdot,\overline{x})\right)(x_3)$.
\end{itemize}
Thus, we can rewrite \eqref{jota3} as $\tresp{\theta\Lambda_1\circ d\varphi(x_1)^{-1}-\Lambda_2\circ
d\varphi(x_2)^{-1}-\varepsilon\Lambda_3\circ d\varphi(x_3)^{-1}}_{\overline{x}}\leq\varepsilon$,
and then $\tresp{\theta\Lambda_1\circ d\varphi(x_1)^{-1}-\Lambda_2\circ
d\varphi(x_2)^{-1}}_{\overline{x}}-\varepsilon\tresp{\Lambda_3\circ d\varphi(x_3)^{-1}}_{\overline{x}}
\leq\varepsilon$. Since $d(\cdot, \overline{x})$ is $1$-Lipschitz, we have $\|\Lambda_3\|_{x_3}\leq 1$.
 Hence,
\begin{equation}\label{vipc}\tresp{\theta\Lambda_1\circ d\varphi(x_1)^{-1}
-\Lambda_2\circ d\varphi(x_2)^{-1}}_{\overline{x}}\leq \varepsilon+\varepsilon\|\Lambda_3\|_{x_3}
\tresp{d\varphi(x_3)^{-1}}_{\overline{x}, x_3}\leq \varepsilon(\varepsilon+2).
\end{equation}
In addition,  we have
\begin{align*}
\|\theta\Lambda_1-\Lambda_2\circ d\varphi(x_2)^{-1}\circ d\varphi(x_1)\|_{x_1}
=&\|[\theta\Lambda_1-\Lambda_2\circ d\varphi(x_2)^{-1}\circ d\varphi(x_1)]\circ
d\varphi(x_1)^{-1}\circ d\varphi(x_1)\|_{x_1}\\ \leq &\tresp{[\theta \Lambda_1
-\Lambda_2\circ d\varphi(x_2)^{-1}\circ d\varphi(x_1)]\circ d\varphi(x_1)^{-1}}_{\overline{x}}
|||d\varphi(x_1)|||_{x_1,\overline{x}} \\ \leq & \tresp{\theta \Lambda_1\circ d\varphi(x_1)^{-1}
-\Lambda_2\circ d\varphi(x_2)^{-1}}_{\overline{x}} (1+\varepsilon) ,\end{align*}
and
\begin{align*}
\|\Lambda_2\|_{x_2}&=\|\Lambda_2\circ [d\varphi(x_2)^{-1}\circ d\varphi(x_1)]\circ
[d\varphi(x_1)^{-1}\circ d\varphi(x_2)]\|_{x_2}
\\&\leq \|\Lambda_2\circ d\varphi(x_2)^{-1}\circ d\varphi(x_1)\|_{x_1}
\tresp{d\varphi(x_1)^{-1}}_{\overline{x},x_1}||| d\varphi(x_2)|||_{x_2,\overline{x}}
\\&\leq \|\Lambda_2\circ d\varphi(x_2)^{-1}\circ d\varphi(x_1)\|_{x_1} (1+\varepsilon)^2.
\end{align*}
Let us check that these inequalities  give us a contradiction.
Since $u$ and $v$ are viscosity solutions, we have $\|\Lambda_1\|_{x_1}\leq 1$ and
$\|\Lambda_2\|_{x_2}\geq 1$.
 Therefore, we can write
\begin{align*}
\tresp{\theta\Lambda_1\circ d\varphi(x_1)^{-1}-\Lambda_2\circ d\varphi(x_2)^{-1}}_{\overline{x}} &
\geq \|\theta \Lambda_1 - \Lambda_2\circ d\varphi (x_2)^{-1}\circ
 d\varphi (x_1)\|_{x_1}(1+\varepsilon)^{-1}
\\&\geq \left(\|\Lambda_2\circ d\varphi(x_2)^{-1}\circ d\varphi(x_1)\|_{x_1}
-\|\Lambda_1\|_{x_1}\theta\right)(1+\varepsilon)^{-1}
\\&\geq \left( \|\Lambda_2\|_{x_2}(1+\varepsilon)^{-2}-\|\Lambda_1\|_{x_1}
\theta\right)(1+\varepsilon)^{-1}
\\&\geq \left( (1+\varepsilon)^{-2}-\theta\right)(1+\varepsilon)^{-1}.
\end{align*}
 Finally,
\begin{equation}\label{vipd}\varepsilon(\varepsilon+2)\ge
\tresp{\theta\Lambda_1\circ d\varphi(x_1)^{-1}-\Lambda_2\circ
d\varphi(x_2)^{-1}}_{\overline{x}}\geq \left( (1+\varepsilon)^{-2}
-\theta\right)(1+\varepsilon)^{-1}.\end{equation}
By letting $\varepsilon\to 0$, we have a contradiction.
\end{proof}


\section{A class of Hamilton-Jacobi equations on Banach-Finsler manifolds}
Let $M$ be a {\bf complete}  and ${C}^1$ Finsler manifold  modeled on a Banach space
with a $C^1$ Lipschitz  bump function and $H\colon M\times \mathbb R\to \mathbb{R}$ be a  continuous function. Recall that we refer to the completeness of $M$ for the Finsler metric $d$.
Let us consider the Hamilton-Jacobi equation
\begin{equation}\label{E1}\tag{E1}
u(x)+H(x,\|du(x)\|_x)=0.
\end{equation}
The aim of this section is to study the existence and  uniqueness of the viscosity solutions
$u\colon M\to \mathbb{R}$ of \eqref{E1}, under certain assumptions.

\begin{Def}
Let us consider a function $u\colon M\to\mathbb{R}$.
\begin{enumerate}
\item[(1)]  The function $u$ is a \textbf{viscosity subsolution of \eqref{E1}} if $u$ is
upper semicontinuous and $u(x)+H(x,\|\Delta\|_x)\leq 0$ for every $x\in M$ and
$\Delta\in D^+u(x)$.
\item[(2)]  $u$ is a
\textbf{viscosity supersolution of \eqref{E1}} if
 $u$ is lower semicontinuous and $u(x)+H(x,\|\Delta\|_x)\geq 0$ for every
$x\in M$ and  $\Delta\in D^-u(x)$.
\item[(3)] $u$ is a \textbf{viscosity solution of \eqref{E1}}
if $u$ is simultaneously a viscosity subsolution and a viscosity supersolution of
\eqref{E1}.
\end{enumerate}
\end{Def}

Let us consider the analogous definition for Finsler manifolds of the
 condition (A) given in \cite[Theorem 3.2]{BZ}  for Banach spaces.

\begin{Def} The  Hamiltonian $H$ in \eqref{E1} satisfies condition (A) whenever there are a constant $C\ge 0$ and a continuous function $\omega: \mathbb R\times \mathbb R\to \mathbb R$  with $\omega(0,0)=0$ such that for any $x_1,x_2\in M$ and any $t_1,t_2\in \mathbb R$,
\begin{equation*}
|H(x_1,t_1)-H(x_2,t_2)|\le \omega(d(x_1,x_2),t_1-t_2)
+C\max\{|t_1|,|t_2|\}d(x_1,x_2).
\end{equation*}
\end{Def}

\begin{Rem} Let us recall that every uniformly continuous Hamiltonian
 $H$ in \eqref{E1} satisfies condition (A). In addition, 
condition (A) implies that $H$ is uniformly continuous in 
$M\times [-K,K]$ for every $K>0$.
\end{Rem}

Let us give now a generalization for Finsler manifolds of the results given in
\cite[Theorem 6.13]{AFLM}, \cite[Theorem 3.2]{BZ}, \cite[Proposition 3.3]{DEL} 
and \cite[Theorem 6.1]{DG}.

\begin{Teo} \label{estacionarias}
Let $M$ be a complete  ${C}^1$ Finsler manifold modeled on a Banach space $X$ with
a $C^1$ Lipschitz bump function
 and let $H\colon M\times \mathbb R\to \mathbb{R}$ be the
Hamiltonian of \eqref{E1}. Assume that $H$ satisfies  condition (A).
If $u$ is a viscosity subsolution and $v$ is a viscosity supersolution of 
\eqref{E1}, both functions
are bounded and for every $x\in M$ either $u$ or $v$ is Lipschitz in a neighborhood of $x$, then
\begin{equation*}
\inf\limits_M (v-u)\geq 0. 
\end{equation*}
\end{Teo}
\begin{proof}
Let us fix $\varepsilon>0$. By applying the Ekeland variational principle to $v-u$,
we can find a point $\overline{x}\in M$ such that
\begin{equation}\label{casiminimo}
\inf_M (v-u) >(v-u)(\overline{x})-\varepsilon
\end{equation}
and
\begin{equation*}
(v-u)(y)\ge (v-u)(\overline{x})-\varepsilon d(y,\overline{x}), \,\, \text{ for all }
 y \in M.
\end{equation*}
Since $(v-u)(y)+\varepsilon d(y, \overline{x})$ attains a minimum
at $\overline{x}$,  $0\in D^-(v-u+\varepsilon d(\cdot, \overline{x}))(\overline{x})$.

Let us assume that $u$ is Lipschitz in a neighborhood of $\overline{x}$ (the other
case is analogous). Thus, there is an open  subset $A\subset M$ with $\overline{x}\in A$ and a constant
$K_{\overline{x}}>0$
such that   $u$  is $K_{\overline{x}}$ --Lipschitz in  $A$.
Let us consider, as we did in Theorem \ref{uniqeik}, the norm
$|||w|||_{\overline{x}}=
||d\varphi^{-1}(\varphi(\overline{x}))(w)||_{\overline{x}} $ for $w\in X$.

Let $(U,\varphi)$ be a chart with $\overline{x}\in U\subset A$ satisfying the Palais
condition for $1+\overline{\varepsilon}$, where
$\overline{\varepsilon}=
\min\{\varepsilon, \varepsilon K_{\overline{x}}^{-1}\}$ such that $\varphi:(U,d)\to 
(\varphi(U), |||\cdot|||_{\overline{x}})$ is 
$(1+\varepsilon)$-bi-Lipschitz (Lemma \ref{desigualdades:BiLipschitz}).

In addition, we consider for  $x\in U$ and
$d\varphi(x):(T_xM, \|\cdot\|_x)\rightarrow (X, \tresp{\cdot}_{\overline{x}})$,
the norms
\begin{align*}
 |||d\varphi(x)|||_{x,\overline{x}}&=\sup \{|||d\varphi(x)(v)|||_{\overline{x}}:
 \, \|v\|_{x}\le 1\},\\
|||d\varphi^{-1}(\varphi(x))|||_{\overline{x},x}&
=\sup\{\|d\varphi^{-1}(\varphi(x))(v)\|_{x}: \, |||v|||_{\overline{x}}\le 1\}.
\end{align*}
We obtained in the proof of Theorem \ref{uniqeik} that, for  $x\in U$,
\begin{equation*}
(1+\overline{\varepsilon})^{-1}\leq |||d\varphi(x)|||_{x,\overline{x}}\leq (1+\overline{\varepsilon}) \quad
\text{and} \quad
(1+\overline{\varepsilon})^{-1}\leq |||d\varphi^{-1}\big(\varphi(x)\big)|||_{\overline{x},x}
\leq (1+\overline{\varepsilon}).
\end{equation*}
Finally, for a linear functional $L:(X,|||\cdot|||_{\overline{x}})\to \mathbb R$,
let us consider the norm
\begin{equation*}
|||L|||_{\overline{x}}=\sup\{|L(v)|:\, |||v|||_{\overline{x}}\le 1\}.
\end{equation*}
Notice that we can consider a Lipschitz extension of $u|_A$ to $M$, denoted by 
 $\tilde{u}:M\to \mathbb R$, in order to apply the local fuzzy rule to 
$v-\tilde{u}+\varepsilon d(\cdot,\overline{x})$.
Thus, by applying  Proposition \ref{fuzzy}, we get  points $x_1,x_2,x_3\in U$ and functionals
$\Delta_1\in D^-(-u)(x_1)$, $\Delta_2\in D^- v(x_2)$ and
$\Delta_3\in D^-(\varepsilon d(\cdot, \overline{x}))(x_3)$ such that
\begin{enumerate}

\item[(i)] $d(x_i,\overline{x})<{\varepsilon}$, for $i=1,2,3$,

\item[(ii)] $|v(x_2)-v(\overline{x})|<{\varepsilon}$ and
$|u(x_1)-u(\overline{x})|<{\varepsilon}$
 and

\item[(iii)]  $|||\Delta_1\circ d\varphi(x_1)^{-1}
+\Delta_2\circ d\varphi(x_2)^{-1}+\Delta_3\circ d\varphi(x_3)^{-1}-0\circ d\varphi(\overline{x})^{-1}|||_{\overline{x}}
< \varepsilon$.

\item[(iv)] $\max \big\{|||\Delta_1 \circ  d\varphi(x_1)^{-1}|||_{\overline{x}}\,,\,
 |||\Delta_2 \circ d\varphi(x_2)^{-1}|||_{\overline{x}}\big\}
\cdot d(x_1,x_2)<\varepsilon$.

\end{enumerate}

Let us denote $\Lambda_1=-\Delta_1 \in D^{+}u(x_1)$, $\Lambda_2= \Delta_2 \in D^{-}v(x_2)$ and $\Lambda_3
 = \varepsilon^{-1} \Delta_3 \in D^{-}(d(\cdot,x_3))(x_3)$.
Then,
\begin{equation} \label{condition3}
|||-\Lambda_1\circ d\varphi(x_1)^{-1}
+\Lambda_2\circ d\varphi(x_2)^{-1}+\varepsilon\Lambda_3\circ d\varphi(x_3)^{-1}
|||_{\overline{x}}
< {\varepsilon}.
\end{equation}
 From \eqref{casiminimo} and condition (ii) we get
\begin{equation}\label{poilh}
\inf\limits_M (v-u)> (v-u)(\overline{x})-\varepsilon> v(x_2)-u(x_1)-3\varepsilon.
\end{equation}
Since $u$ is a viscosity subsolution of \eqref{E1} and $v$ is a viscosity supersolution of \eqref{E1},
we get
\begin{align}\label{subsup}
-u(x_1) &\geq H(x_1,\|\Lambda_1\|_{x_1}),\\
v(x_2)&\geq -H(x_2,\|\Lambda_2\|_{x_2}). \notag
\end{align}
Consequently, by  inequalities \eqref{poilh} and \eqref{subsup},
\begin{align}\notag
\inf\limits_M (v-u)  &
>H(x_1,\|\Lambda_1\|_{x_1})-H(x_2,\|\Lambda_2\|_{x_2})-3\varepsilon \ge
\\ &\ge -\big[\omega(d(x_1,x_2), \|\Lambda_1\|_{x_1}-\|\Lambda_2\|_{x_2})+
C\max\{\|\Lambda_1\|_{x_1}, \|\Lambda_2\|_{x_2}\}d(x_1,x_2)\big]
-3\varepsilon,\label{7b}
\end{align}
\noindent where $\omega$ is the  function and $C\ge 0$ is the constant given in condition (A) for $H$. Now, inequality \eqref{condition3} above yields
\begin{equation*}
\Big| |||\Lambda_1\circ d\varphi(x_1)^{-1}|||_{\overline{x}}
-|||\Lambda_2\circ d\varphi(x_2)^{-1}|||_{\overline{x}}\Big|\le
 |||\varepsilon\Lambda_3\circ d\varphi(x_3)^{-1}|||_{\overline{x}}+ {\varepsilon}.
\end{equation*}
\noindent Recall that the function $d(\cdot,\overline{x})$ is 1-Lipschitz and thus
$||\Lambda_3||_{x_3}\le 1$. Therefore,
\begin{equation*}
|||\varepsilon\Lambda_3\circ d\varphi(x_3)^{-1}|||_{\overline{x}}\le
\varepsilon \|\Lambda_3\|_{x_3}|||d\varphi(x_3)^{-1}|||_{\overline{x},x_3}
\le \|\Lambda_3\|_{x_3}\varepsilon(1+
\varepsilon)\le \varepsilon(1+{\varepsilon}).
\end{equation*}
 Now,
\begin{equation*}
|||\Lambda_2\circ d\varphi(x_2)^{-1}|||_{\overline{x}}
\ge \|\Lambda_2\|_{x_2}|||d\varphi(x_2)|||_{x_2,\overline{x}}^{-1}
\ge \|\Lambda_2\|_{x_2} (1+\overline{\varepsilon})^{-1}
\end{equation*}
 and
\begin{equation*}
\tresp{\Lambda_1\circ d\varphi(x_1)^{-1}}_{\overline{x}}\le \|\Lambda_1\|_{x_1}
\tresp{d\varphi(x_1)^{-1}}_{\overline{x}, x_1}\le \|\Lambda_1\|_{x_1} (1+\overline{\varepsilon}).
\end{equation*}
\noindent Therefore,
\begin{equation*}
- \|\Lambda_1\|_{x_1} (1+\overline{\varepsilon})+ \|\Lambda_2\|_{x_2} (1+\overline{\varepsilon})^{-1}
\le\varepsilon(2+\varepsilon).
\end{equation*}
\noindent  Since $u$ is $K_{\overline{x}}$ --Lipschitz in  $U$, we have $||\Lambda_1||_{x_1}
\le K_{\overline{x}}$\, and, by computing,
we obtain
\begin{align*}
\|\Lambda_2\|_{x_2}- \|\Lambda_1\|_{x_1} &<\varepsilon(2+\varepsilon)
+\overline{\varepsilon}\|\Lambda_1\|_{x_1}
+\frac{\overline{\varepsilon}}{1+\overline{\varepsilon}} \|\Lambda_2\|_{x_2}
 \\ &\le \varepsilon(2+\varepsilon)
+ \overline{\varepsilon}\|\Lambda_1\|_{x_1}
+\overline{\varepsilon}\bigl( \varepsilon(2+\varepsilon)+(1+\overline{\varepsilon})
 \|\Lambda_1\|_{x_1}\bigr) \\
 &\le   \varepsilon(4+4\varepsilon+\varepsilon^2).
\end{align*}
In an analogous way we obtain 
$\|\Lambda_1\|_{x_1}- \|\Lambda_2\|_{x_2} <
 \varepsilon(4+4\varepsilon+\varepsilon^2)$. Also, condition (iv) yields

\begin{align}\notag
\varepsilon & >
\max \big\{|||\Delta_1 \circ  d\varphi(x_1)^{-1}|||_{\overline{x}}\,,\,
 |||\Delta_2 \circ d\varphi(x_2)^{-1}|||_{\overline{x}}\big\}
\cdot d(x_1,x_2) \ge \\  \label{fuzzyplus} 
&\ge (1+\varepsilon)^{-1}\max \big\{||\Delta_1 ||_{x_1}\,,\,
 ||\Delta_2 ||_{x_2}\big\}\cdot d(x_1,x_2).
\end{align}
In addition,  $d(x_1,x_2)<2\varepsilon$ and, by the  continuity of
$\omega$ and inequality \eqref{fuzzyplus}, we obtain
 \begin{equation*}
\omega \big(d(x_1,x_2), \|\Lambda_1\|_{x_1}-\|\Lambda_2\|_{x_2}\big)+
C\max \big\{\|\Lambda_1\|_{x_1}, \|\Lambda_2\|_{x_2}\big\} \cdot d(x_1,x_2)\to 0 \text{ as }
 \varepsilon\to 0.
\end{equation*}
 Finally, inequality \eqref{7b} yields $\inf_{M}(v-u)\ge 0$.
\end{proof}
\begin{Rem} \begin{enumerate} \label{remarksestacionarias}
\item If we assume in the assumptions of Theorem \ref{estacionarias} that  either $u$ or $v$ is $L$-Lipschitz, then it is enough to assume that the Hamiltonian $H$ is uniformly continuous in $ M\times [0,R]$ for some $R>L$.
\item  It is worth noticing that Theorem \ref{estacionarias} holds (with few modifications in the proof) for the weaker condition
on $H$ denoted as (*) 
 in \cite{DG}: a Hamiltonian $H$ of \eqref{E1}
verifies condition (*) if
\begin{align*}
|H(x_1,t)-H(x_2,t)|\to 0 &\text{ \ \ as \ \ }  d(x_1,x_2)(1+|t|)\to 0  \text{ \ \  \ uniformly on 
\   } t \in \mathbb  R, \, x_1, x_2 \in M \text{ and }
\\
|H(x,t_1)-H(x,t_2)|\to 0 &\text{ \ \ as \ \ }  |t_1-t_2|\to 0   \text{  \ \ \  uniformly on \  } 
 x \in M, \, t_1, t_2 \in \mathbb R. 
\end{align*}
\end{enumerate}
\end{Rem}

A few modifications of Theorem \ref{estacionarias} yield the following results on the stability  of the viscosity solutions.


\begin{Pro}\label{estabilidad1}
Let $M$ be a complete  ${C}^1$ Finsler manifold modeled on a Banach space with
a $C^1$ Lipschitz bump function
 and let $H_1,\, H_2\colon M\times \mathbb R\to \mathbb{R}$ be two
Hamiltonians of \eqref{E1}. Assume that $H_1$ and $H_2$ verify condition (A).
If $u$ is a viscosity subsolution of \eqref{E1} for the Hamiltonian $H_1$ and $v$ is a viscosity supersolution of \eqref{E1} for the Hamiltonian $H_2$, the functions $u$ and $v$
are bounded and for every $x\in M$ either $u$ or $v$ is Lipschitz in a neighborhood of $x$, then
\begin{equation*}
\sup\limits_M(u-v)\le \sup\limits_{M\times \Real}(H_2-H_1).
\end{equation*}
\end{Pro}


An immediate consequence of Proposition \ref{estabilidad1} is the next result. First, let us recall the definition of a equi-continuous family of functions.


\begin{Def}
Let $\Gamma$ be a topological space and let $S$ be an arbitrary set. A family of functions $\{f_\gamma : S\to \Real\}_{\gamma\in \Gamma}$ is equi-continuous at $\gamma_0\in \Gamma$ if for every $\varepsilon>0$ there exists an open neighborhood $U$ of $\gamma_0$ such that $|f_\gamma(s)-f_{\gamma_0}(s)|<\varepsilon$ for all $\gamma \in U$ and $s\in S$. A family $\{f_\gamma: \gamma\in \Gamma\}$ is equi-continuous if  it is equi-continuous at every $\gamma_0\in \Gamma$.
\end{Def}

\begin{Cor}
Let $M$ be a complete  ${C}^1$ Finsler manifold modeled on a Banach space with
a $C^1$ Lipschitz bump function, and $\Gamma$ a topological space. Let  $H_\gamma: M\times \Real \to \Real$  be  Hamiltonians of \eqref{E1} satisfying condition (A) for all $\gamma\in \Gamma$. Let us assume that:
\begin{enumerate}
\item the family of functions $\{H_\gamma: \gamma\in \Gamma\}$ is equi-continuous and
\item  for every  $\gamma\in \Gamma$, the function $u_\gamma:X\to \Real$ is a locally Lipschitz viscosity solution of \eqref{E1} for the Hamiltonian $H_\gamma$.
\end{enumerate}
Then, for every $\gamma_0 \in \Gamma$ and every $\varepsilon>0$, there exists  an open neighborhood $U$ of $\gamma_0$ such that $||u_\gamma - u_{\gamma_0}||_\infty=\sup\{|u_\gamma(x) - u_{\gamma_0}(x)|:\,x\in M\}<\varepsilon$ for all $\gamma\in U$.
\end{Cor}

In the  following results, we adapt  Perron's method to Finsler manifolds and, in particular to prove the existence and uniqueness of the bounded viscosity solutions on a  class of   Hamilton-Jacobi equations of the form  \eqref{E1}. Let us consider the more general class of Hamilton-Jacobi equations of the form
\begin{equation} \label{E2}\tag{E2}
F(x,du(x),u(x))=0,
\end{equation}
where $F: TM^* \times \mathbb R\to \mathbb R$ is a continuous Hamiltonian.
Let us recall that  the topology of $TM^* $ satisfies the first axiom of countability:  for each point $(x,\Lambda)\in TM^* $ and  a fixed chart $(U,\psi)$ such that $x\in U$, the family
\begin{equation*}
U^n_n(x):= \{ (y,\Delta) \in TU^*:\,d(y,x)<\frac{1}{n},\, \text{ and }
 || \Lambda \circ d\psi(x)^{-1} - \Delta \circ d\psi(y)^{-1}  || < \frac{1}{n} \} 
\end{equation*} 
is  a countable neighborhood basis of $(x,\Lambda)$.
Also, a sequence $\{(x_n, \Delta_n)\}_{n\in \mathbb N}\subset TM^*$ converges to $(x,\Delta)$ in  $TM^*$ iff
\begin{enumerate}
\item $\lim_{n\to \infty} d(x_n,x)=0$ and
\item $\lim_{n\to \infty}||\Delta_n\circ d\varphi(x_n)^{-1}-\Delta\circ d\varphi(x)^{-1}||= 0$ for every
chart $(U,\varphi)$ on $M$ with $x\in U$. Equivalently, there is a chart $(U,\varphi)$ on $M$ with $x\in U$ such that $\lim_{n\to \infty}||\Delta_n\circ d\varphi(x_n)^{-1}-\Delta\circ d\varphi(x)^{-1}||= 0$.
 (Let us recall that, in general, we assume 
$\Delta_n\circ d\varphi(x_n)^{-1}$  defined only for $n\ge n_0$, where  $n_0$ depends on the chart $(U, \varphi)$).
\end{enumerate}
Notice that we can define the continuity of $F$ (given in \eqref{E2}) in terms
of sequences: the Hamiltonian $F$ is continuous at $(x,\Delta,t)\in TM^*\times \mathbb R$ if  $\lim_{n\to \infty}F(x_n, \Delta_n,t_n)=F(x, \Delta,t)$
for every sequence $\{(x_n, \Delta_n,t_n)\}_{n\in \mathbb N}\subset TM^*\times \mathbb R$ with limit
$(x,\Delta,t)$.

It can be easily  checked that condition (2) above implies $ \lim_n ||\Delta_n||_{x_n} = ||\Delta||_x$ and thus, for a continuous function $H:M\times \mathbb R\to \mathbb R$,  the Hamilton-Jacobi equation considered in \eqref{E1}
\begin{equation*}
F(x,du(x),u(x)):=u(x)+H(x,||du(x)||_x)=0
\end{equation*}
is a particular case of \eqref{E2}.
Let us recall that a function $u\colon M\to\mathbb{R}$
\begin{enumerate}
\item[(1)]  is a \textbf{viscosity subsolution of \eqref{E2}} if $u$ is
upper semicontinuous and $F(x,\Delta,u(x))\leq 0$ for every $x\in M$ and
$\Delta\in D^+u(x)$,
\item[(2)]   is a
\textbf{viscosity supersolution of \eqref{E2}} if
 $u$ is lower semicontinuous and $F(x,\Delta,u(x))\geq 0$ for every
$x\in M$ and  $\Delta\in D^-u(x)$,
\item[(3)]  is a \textbf{viscosity solution of \eqref{E2}}
if $u$ is simultaneously a viscosity subsolution and a viscosity supersolution of
\eqref{E2}.
\end{enumerate}

\begin{Lem}\label{envelope}
Let $M$ be a ${C}^1$ Finsler manifold modeled on a Banach space with
a $C^1$ Lipschitz bump function.  Let $\Omega$ be an open subset of $M$. Let $\mathcal{F}$ be a locally uniformly bounded  family of  functions from $\Omega$ into $\mathbb R$ and $u=\sup\{v:\, v\in \mathcal{F}\}$ on $\Omega$. If every $v \in \mathcal{F}$ is a viscosity subsolution of \eqref{E2} on $\Omega$, where the Hamiltonian $F:T\Omega^*\times \Real \to \Real$ is continuous, then $u^*$ is also  a viscosity subsolution of \eqref{E2} on $\Omega$.
\end{Lem}
\begin{proof}
Let us consider $x\in \Omega$ and $\Delta \in D^+u^*(x)$. By Proposition \ref{conv} (stability of the superdifferentials) there exist  sequences  $\{v_n\}$ in $\mathcal{F}$ and $\{(x_n,\Delta_n)\}_{n\in \mathbb N}$ in $TM^*$ with $x_n \in \Omega$ and $\Delta_n \in D^+v_n(x_n)$ for every $n\in \mathbb N$, such that
\begin{itemize}
\item[(i)] $\lim_{n\to \infty} v_n(x_n)=u^*(x)$, and
\item[(ii)] $\lim_{n\to \infty}  (x_n,\Delta_n) =(x,\Delta)$ in  $TM^*$ (i.e. $\lim_{n\to \infty} d(x_n,x)=0$ and
$\lim_{n\to \infty}||\Delta_n\circ d\varphi(x_n)^{-1}-\Delta\circ d\varphi(x)^{-1}||= 0$ for every
chart $(U,\varphi)$ with $x\in U$).
\end{itemize}

Since $v_n$ is a viscosity subsolution of \eqref{E2} on $\Omega$ for every $n\in \mathbb N$, we have $F(x_n,\Delta_n, v_n(x_n))\le 0$ for every $n\in \mathbb N$. Hence,  $F(x,\Delta,u^*(x))\le 0$ and $u^*$ is a viscosity subsolution of \eqref{E2} on $\Omega$.
\end{proof}
\begin{Rem}
In particular, in the above context, the supremum of two viscosity subsolutions of  \eqref{E2} on $\Omega$ is a viscosity subsolution of \eqref{E2} on $\Omega$.
\end{Rem}
\begin{Pro}\label{familias de subsoluciones}
Let $M$ be a  ${C}^1$ Finsler manifold modeled on a Banach space with
a $C^1$ Lipschitz bump function.  Let $\Omega$ be an open subset of $M$ and let $F:T\Omega^*\times \mathbb R\to \mathbb R$ be a continuous Hamiltonian  on $\Omega$. Assume that there are two continuous functions $s_0,s_1:\Omega \to \mathbb R$, which are respectively a viscosity subsolution and a viscosity supersolution of \eqref{E2} on $\Omega$ and $s_0\le s_1$ on $\Omega$. Let us define the family
 \begin{equation*}
 \mathcal F=\{w:\Omega\to \mathbb R:\, s_0\le w\le s_1 \text{ on } \Omega \text{ and }
 w \text{ is  a viscosity subsolution of } \eqref{E2} \text{ on } \Omega\},
 \end{equation*}
and the function $u=\sup \mathcal F$. Then, $u^*$ is a viscosity subsolution of \eqref{E2} on $\Omega$, $(u^*)_*$ is a viscosity supersolution of \eqref{E2} on $\Omega$ and  $s_0\le (u^*)_*\le u^*\le s_1$.
\end{Pro}
\begin{proof} The proof is similar to the one given in \cite[Theorem 6.4]{DG}. We shall give it here for completeness. Notice that, since $s_0$ and $s_1$ are continuous  in $\Omega$,  the family $\mathcal F$ is locally bounded on $\Omega$. Thus, by Lemma \ref{envelope}, $u^*$ is a viscosity subsolution of \eqref{E2} on $\Omega$.

Let us suppose that $v=(u^*)_*$ is not a viscosity supersolution of \eqref{E2}. Then, there exist $x_0\in M$ and $\Delta_0\in D^- v(x_0)$ such that $F(x_0,\Delta_0,v(x_0))<0$. According to the definition of the subdifferential, there is a $C^1$ smooth function $g:M\to \Real$ such that $v-g$ attains a local minimum at $x_0$ and $\Delta_0 = dg(x_0)$. Then, there exists  an open neighborhood  $U$ of $x_0$, where $v(x)-g(x)\ge v(x_0)-g(x_0)$ for all $x\in U$.  Notice that $\widetilde{g}(x)=g(x)+v(x_0)-g(x_0)$ is also a $C^1$ smooth function with $\Delta_0 = d\widetilde{g}(x_0)$ and $v-\widetilde{g}$ attains a local minimum at $x_0$, and thus we may assume
\begin{equation}\label{cond:subdif}
F(x_0, dg(x_0),v(x_0))<0, \quad v(x_0)=g(x_0) \quad \text{and} \quad g(x)\le v(x) \quad \text{for all $x\in U$.}
\end{equation}

It is clear that $g\le v\le  s_1$ on $U$.  Let us check that, in fact,
 $g(x_0)<s_1(x_0)$. Indeed, otherwise $s_1-g$ would attain a local minimum at $x_0$ and thus
   $dg(x_0)\in D^-s_1(x_0)$. Since $s_1$ is a viscosity supersolution,
   $0\le F(x_0,dg(x_0),s_1(x_0))= F(x_0,dg(x_0),v(x_0)) <0$, which is a contradiction.

Since $M$ is modeled on a Banach space $X$ with a $C^1$ Lipschitz bump function, we can choose  $\delta>0$ and a $C^1$ Lipschitz bump function $b:M \to [0,1]$ with
\begin{itemize}
\item[(1)] $B(x_0,2\delta) \subset U$,
\item[(2)] $b(x_0)>0$,
\item[(3)]  $b(x)=0$ whenever $d(x,x_0)\ge \delta$, and
\item[(4)] $\sup \{|b(x)|:x\in M\}$ and $\sup\{||db(x)||_x: x\in M\}$ small enough so that
\begin{align*}\label{cond:subdif:2}
& F(x,dg(x)+db(x),g(x)+b(x))<  0 \quad \text{whenever $d(x,x_0)<2\delta$, and}\notag \\
& \quad g(x)+b(x)\le  s_1(x) \quad \text{for every $x\in U$.}
\end{align*}
\end{itemize}
Clearly,  $g+b$ is a viscosity subsolution of \eqref{E2} on $B(x_0,2\delta)$.
Now, define
 \begin{equation*}
w(x)=\left\{ \begin{array}{ll} \max\{g(x)+b(x),u^*(x)\} &\mbox{ for all }x\in B(x_0,2\delta),\\
u^*(x) &\mbox{ for all }x\in \Omega\setminus B(x_0,2\delta).\end{array}\right.
\end{equation*}
On the one hand,  $u^*(x)\ge v(x)\ge g(x)=g(x)+b(x)$ for all $x\in U\setminus \overline{B}(x_0,\delta)$. Therefore,
 $w(x)=u^*(x)$ for all $x\in \Omega_1:=\Omega\setminus \overline{B}(x_0,\delta)$ and then, $w$ is a viscosity subsolution of \eqref{E2} on $\Omega_1$. On the other hand, $w$ is the supremum of two viscosity subsolutions on $\Omega_2:=B(x_0,2\delta)$. Thus $w$ is a viscosity subsolution of \eqref{E2} on $\Omega_2$, and consequently it is a viscosity subsolution on $\Omega=\Omega_1\cup \Omega_2$.

 Since $s_0\le w\le s_1$, we have $w\in \mathcal{F}$ and then,  $w\le u\le u^*$ on 
$\Omega$, and $u^*(x)\ge w(x)\ge g(x)+b(x)$ on $B(x_0, \delta)$. Therefore, $v(x)=(u^*)_*(x)\ge g(x)+b(x)$ on $B(x_0, \delta)$. In particular,
 $v(x_0)=(u^*)_*(x_0)\ge g(x_0)+b(x_0) > g(x_0)$
  which contradicts \eqref{cond:subdif}.
\end{proof}
\begin{Cor} \label{corolarioexistenciaFinsler}
Let $M$ be a complete ${C}^1$ Finsler manifold modeled on a Banach space with
a $C^1$ Lipschitz bump function. Let $H:M\times \Real \to \Real$ be the Hamiltonian of \eqref{E1}. Assume that  there are constants $K_0,K_1>0$ such that $K_0 \le H(x,0) \le K_1$ for all $x\in M$ and at least one of the following conditions holds:
\begin{itemize}
\item[(i)] $H$ is uniformly continuous and $\liminf_{t\to \infty} H(x,t)>K_1$ for each $x\in M$.

\item[(ii)] $H$ satisfies condition (A), \,there is a constant $K_1'$ such that $\liminf_{t\to \infty} H(x,t)\ge K_1'>K_1$ for each $x\in M$ and the limit is locally uniform on $M$.
\end{itemize}
Then, there exists a unique bounded viscosity solution $u$ of the equation \eqref{E1}. Moreover, if we define the family
 \begin{equation*}
 \mathcal{F}:= \{w: M\to \Real: \, -K_1 \le w\le -K_0 \text{ on $M$, and $w$ is a viscosity subsolution of \eqref{E1}}\},
 \end{equation*}
then, the viscosity solution is $u=\sup\{w:w\in \mathcal{F}\}$  and $u$ is locally Lipschitz.
\end{Cor}
\begin{proof}
It can be easily checked that the functions $s_0(x)=-K_1$ and $s_1(x)=-K_0$ are respectively a viscosity subsolution and a viscosity
supersolution of \eqref{E1}. Let us take $u^*$ the upper semicontinuous envelope of $u:=\sup\{w:w\in \mathcal{F}\}$, and  $(u^*)_*$ the lower semicontinuous envelope of $u^*$. By Proposition \ref{familias de subsoluciones}, $u^*$ and $(u^*)_*$ are respectively a viscosity subsolution and a viscosity supersolution of \eqref{E1} and
$-K_1\le (u^*)_*\le u^*\le -K_0$.

Notice that if $w$ is a viscosity subsolution of \eqref{E1} with $-K_1 \le w\le -K_0$ in $M$, then $H(x,||\Delta||_x)\le -w(x)\le K_1$ for each $x\in M$ and $\Delta \in D^+ w(x)$. Let us fix $x\in M$.  Since $H$  satisfies either condition $(i)$ or $(ii)$ above, there are constants $r_x, R_x>0$ (depending only on $H$ and $x$) such that $ H(z,t)>K_1$ whenever $z\in B(x,r_x)$ and $t>R_x$. Therefore
 $||\Delta||_z\le R_x$ for all $z\in B(x,r)$ and $\Delta \in D^+ w(z)$. By applying Theorem \ref{devillemean}, we conclude that $-w$ is $R_x$--Lipschitz in $B(x,\frac{r_x}{4})$, and so is $w$.

 This implies that the function $u=\sup\{w:w\in \mathcal{F}\}$ satisfies the same Lipschitz condition: $u$ is  $R_x$--Lipschitz in $B(x,\frac{r_x}{4})$.
Thus, by the definition of upper and lower semicontinuous envelopes, we have $u=u^*=(u^*)_*$. This yields $u=\sup\{w:w\in \mathcal{F}\}$ is a  bounded and locally Lipschitz viscosity solution of \eqref{E1}.

 Finally, if $g:M\to \mathbb R$ is a bounded viscosity solution of \eqref{E1},  according  to Theorem \ref{estacionarias}, necessarily $g=u$. This provides the uniqueness of the bounded viscosity solution of \eqref{E1} and finishes the proof.
\end{proof}

\begin{Rem}
Notice that a uniformly continuous  Hamiltonian $H:M\times \Real \to \Real$ of \eqref{E1} satisfies condition $(i)$ given in Corollary 
\ref{corolarioexistenciaFinsler} whenever $H(x,\cdot)$ is coercive for each $x\in M$, i.e. $\lim_{t\to \infty} H(x,t)=+\infty$ for each $x\in M$. Also a Hamiltonian $H$
of \eqref{E1} with property (A) satisfies condition $(ii)$ given in Corollary \ref{corolarioexistenciaFinsler} whenever $H$ is uniformly coercive in $M$, i.e. 
$\lim_{t\to \infty} H(x,t)=+\infty$ uniformly on $ M$.
\end{Rem}

\begin{Exs} Let us consider some examples regarding Corollary \ref{corolarioexistenciaFinsler}. Recall that $M$ is  a complete $C^1$ Finsler manifold modeled on a Banach space with a $C^1$ Lipschitz bump function.
\begin{enumerate}
\item[(1)] Let us consider the Hamilton-Jacobi equation
\begin{equation*}
u(x)+\min\{||du(x)||_x,a\}-\cos d(x_0,x)=0,
\end{equation*}
where $a>2$ is a fixed real number and $x_0$ is a fixed point in the Finsler manifold $M$.
The Hamiltonian
$H:M\times \mathbb R \to \mathbb R, \,H(x,t)=\min\{t,a\}-\cos d(x_0,x)$ is uniformly continuous.
Moreover,
$-1\le H(x,0)=-\cos d(x_0,x) \le 1$ for $ x\in M$,
and $\lim_{t\to  \infty}H(x,t)=a-\cos  d(x_0,x)\ge a-1>1$,
uniformly in $x\in M$. By Corollary \ref{corolarioexistenciaFinsler}, there is a unique bounded  viscosity solution $u$ such that $-1\le u\le 1$. Moreover, if $t\ge a$ then $H(x,t)>1$. Thus, every superdifferential of $u$ is bounded above by $a$ and $u$ is $a$-Lipschitz.

\item[(2)] Let us consider the Hamilton-Jacobi equation
\begin{equation*}
u(x)+||du(x)||_x-\cos d(x_0,x)=0.
\end{equation*}
The Hamiltonian
$
H(x,t)= t-\cos  d(x_0,x)
$
is uniformly continuous,
$
-1\le H(x,0)\le 1
$
for all $x\in M$ and
$
\lim_{t\to  \infty}H(x,t)=\infty
$
uniformly in $ M$.
By Corollary \ref{corolarioexistenciaFinsler},  there is a unique bounded
viscosity solution  $u$, which is locally Lipschitz and  $-1\le u\le 1 $.
Moreover, if $t>2$ then
 $H(x,t)>1$. Thus, the superdifferentials of $u$ are bounded by $2$ and $u$ is $2$-Lipschitz.

\item[(3)] For  $0<a<b$, let us consider the Hamilton-Jacobi equation
\begin{equation*}
u(x)+\min\{||du(x)||_x,1\}-\frac{a+d(x_0,x)}{b+d(x_0,x)}=0.
\end{equation*}
The Hamiltonian
$
H(x,t)=\min\{t,1\}-\frac{a+d(x_0,x)}{b+d(x_0,x)}
$
is uniformly continuous,
$
-1\le H(x,0)=-\frac{a+d(x_0,x)}{b+d(x_0,x)}\le -\frac{a}{b}
$
for all $x\in M$ and
$
\lim_{t\to \infty}H(x,t)=1-\frac{a+d(x_0,x)}{b+d(x_0,x)}>0
$
for every $x\in M$. By Corollary \ref{corolarioexistenciaFinsler}, there is a unique bounded viscosity solution $u$,
 which is locally Lipschitz and $\frac{a}{b}\le u\le 1$. Notice that, if $\min\{t,1\}>1-\frac{a}{b}$, then $H(x,t)>-\frac{a}{b}$. Therefore, the norm of the superdifferentials of $u$ are bounded above by $1-\frac{a}{b}$ and
thus $u$ is $(1-\frac{a}{b})$-Lipschitz.

\item[(4)] Let us consider the Hamilton-Jacobi equation
\begin{equation*}
u(x)+\frac{1+2||du(x)||_x}{1+||du(x)||_x+d(x_0,x)}=0.
\end{equation*}
The Hamiltonian
$H(x,t)=\frac{1+2|t|}{1+|t|+d(x_0,x)}$
is uniformly continuous. In addition,
$0\le H(x,0)=\frac{1}{1+d(x_0,x)}\le 1$ for all $x\in M$ and
$\lim_{t \to \infty}H(x,t)=\frac{1+2|t|}{1+|t|+d(x_0,x)}=2$ for  every $x\in M$.
Moreover,  it can be easily checked that for every $x\in M$, if $t>d(x_0,x)$ then $H(x,t)>1$.
Therefore, by Corollary \ref{corolarioexistenciaFinsler},
there is a unique bounded viscosity solution $u$, which is locally Lipschitz and $-1\le u\le 0$.
Moreover, for every $R>0$,  $u$ is $R$-Lipschitz in $B(x_0,\frac{R}{4})$.

\item[(5)]  A generalization of the  example (2) is the Hamilton-Jacobi equation
\begin{equation*}
u(x)+||du(x)||_x-f(x)=0,
\end{equation*}
where $f:M\to \mathbb R$ is uniformly continuous and bounded.
The Hamiltonian
$
H(x,t)= t-f(x)
$
is uniformly continuous,
$
K_0:=\inf_M f\le H(x,0)=f(x)\le\sup_M f:=K_1
$
for all $x\in M$ and
$
\lim_{t\to  \infty}H(x,t)=\infty
$
uniformly in $ M$.
By Corollary \ref{corolarioexistenciaFinsler},  there is a unique bounded
viscosity solution  $u$, which is locally Lipschitz and  $-K_1\le u\le K_0 $.
Moreover, if $t>K_1-K_0$ then
 $H(x,t)>K_1$. Thus, the superdifferentials of $u$ are bounded by $K_1-K_0$ and $u$ is $(K_1-K_0)$-Lipschitz.

\end{enumerate}

\end{Exs}

\section{A class of evolution Hamilton-Jacobi equations on Banach-Finsler manifolds}
Let $M$ be a {\bf complete}   ${C}^1$ Finsler manifold
modeled on a Banach space with a ${C}^1$ Lipschitz bump function.
Let us consider a  continuous function $H:
[0,\infty) \times M  \times \mathbb R  \to \mathbb{R}$ and
the  Hamilton-Jacobi equation
\begin{equation}\label{E3}\tag{E3}
\begin{cases}
u_t+H(t,x,\|u_x\|_x)=0, \quad (t>0)\\
u(0,x)=h(x),
\end{cases}
\end{equation}
where $u: [0,\infty)\times M\to\mathbb{R}$ and
$h: M\to\mathbb{R}$ is the initial condition which we assume to be bounded
and continuous.

\begin{Def}
Let us consider a function $u:[0,\infty)\times M\to\mathbb{R}$.
\begin{enumerate}
\item[(1)]  $u$ is a \textbf{viscosity subsolution of \eqref{E3}}
 if  $u$ is upper semicontinuous,
 $\alpha+H(t,x,\|\Delta\|_x)\leq 0$  for every
 $(\alpha,\Delta)\in D^+u(t,x)$ and $(t,x)\in \mathbb R^+\times M$ and $u(0,x)\leq h(x)$ for every
 $x\in M$.
 \item[(2)]  $u$ is a
  \textbf{viscosity supersolution of \eqref{E3}} if
   $u$ is lower semicontinuous, $\alpha+H(t,x,\|\Delta\|_x)\ge 0$  for every
 $(\alpha,\Delta)\in D^-u(t,x)$ and $(t,x)\in \mathbb R^+\times M$ and $u(0,x)\ge h(x)$ for every
 $x\in M$.
   \item[(3)] $u$ is a \textbf{viscosity solution of \eqref{E3}}
  if  $u$ is simultaneously a
  viscosity subsolution and a viscosity supersolution of \eqref{E3}.
  \end{enumerate}
\end{Def}
Let us consider the analogous
 condition (A) for Hamiltonians of \eqref{E3}.
\begin{Def} The  Hamiltonian $H$ of \eqref{E3} satisfies condition (A) whenever there are a constant $C\ge 0$ and a continuous function $\omega: [0,\infty)\times \mathbb R \times \mathbb R\to \mathbb R$  with $\omega(0,0,0)=0$ such that for any $t_1,t_2\in [0,\infty)$, $x_1,x_2\in M$ and  $r_1,r_2\in \mathbb R$,
\begin{equation*}
|H(t_1,x_1,r_1)-H(t_2,x_2,r_2)|\le \omega(|t_1-t_2|,d(x_1,x_2),r_1-r_2)
+C\max \big\{|r_1|,|r_2|\big\}\big(|t_1-t_2|+d(x_1,x_2)\big).
\end{equation*}
\end{Def}

\begin{Rem} Let us recall that every uniformly continuous Hamiltonian
 $H$ of \eqref{E3} satisfies condition (A). In addition, 
condition (A) implies that $H$ is uniformly continuous in 
$[0,\infty)\times M\times [-K,K]$ for every $K>0$.
\end{Rem}

In the next result we follow the ideas of  \cite{BZ}, \cite[Theorem 6.2]{DG},
\cite{AFLM2},  \cite{DEL}, \cite{EH} and \cite{IShort}
to obtain a generalization for  Finsler manifolds.

\begin{Teo} \label{comparisonevolution}
Let $M$ be a complete and  ${C}^1$ Finsler manifold modeled on a Banach space with a
${C}^1$ Lipschitz bump function and let $H\colon [0,\infty)\times  M \times \mathbb R \to \mathbb{R}$
 be the Hamiltonian of $\eqref{E3}$.
 Assume that $H$ verifies condition (A).
  If $u$ is a viscosity subsolution  and $v$ is a viscosity
  supersolution of \eqref{E3}, for every $T>0$ both functions are bounded in $[0,T)\times M$ and  for every $(t,x)\in (0,\infty) \times M$ either $u$ or $v$ is Lipschitz  in a neighborhood  of $(t,x)$, then
\begin{equation*}
\inf_{[0,\infty)\times M}(v-u)\ge 0.
\end{equation*}
\end{Teo}
\begin{proof}
Assume, by contradiction, that
there is $(a,z)\in (0,\infty)\times M$ such that $v(a,z)-u(a,z)<0$.
Let us fix  $T>0$ large enough so that $\inf_{[0,T)\times M}(v-u)< 0$.
For $\delta>0$, let us set
\begin{equation*}
u_\delta(t,x)=u(t,x)-\frac{\delta}{T-t}, \qquad
(t,x)\in [0,T)\times M.
\end{equation*}
It is easy to check that $u_\delta$ is a viscosity subsolution and $v$ is a viscosity supersolution of
\begin{equation*}
u_{t}(t,x)+H(t,x,||u_{x}(t,x)||_x)=0, \qquad (t,x)\in[0,T)\times M
\end{equation*}
with initial condition $u_{\delta}(0,x) +\frac{\delta}{T}\le h(x) \le v(0,x)$ for $x\in M$. Let us fix  $\delta>0$ small enough so that $\inf_{(0,T)\times M}(v-u_\delta)<0<\frac{\delta}{T}\le\inf_{\{0\}\times M}(v-u_\delta).$
Moreover, the boundedness of $v-u$ in $[0,T]$ yields the existence of  $0<T'<T$  such that
\begin{equation*}\inf_{(0,T')\times M}(v-u_\delta)<0<\inf_{\{0,T'\}\times M} (v-u_\delta).
\end{equation*}
Thus, we may assume  $u\equiv u_\delta$  and $v$ are a viscosity subsolution and a viscosity supersolution respectively  of
\begin{equation*}
u_{t}(t,x)+H(t,x,||u_x(t,x)||_x)=0, \qquad (t,x)\in[0,T')\times M
\end{equation*}
with initial condition
\begin{equation*}
u(0,x) +\frac{\delta}{T}\le h(x) \le v(0,x), \quad  x\in M,
\end{equation*}
where
$u$ and $v$ are bounded in $[0,T']\times M$,  for every $(t,x)\in (0,T') \times M$ either $u$ or $v$ is Lipschitz  in a neighborhood  of $(t,x)$ and
\begin{equation}\label{cotas}
\inf_{(0,T')\times M}(v-u)<0<\inf_{\{0,T'\}\times M} (v-u)
\end{equation}

Let us fix $\eta>0$  small enough so that
 $\varphi:\mathbb{R}\times M\to \mathbb R$ defined as
\begin{equation*}
\varphi(t,x)=\begin{cases} v(t,x)-u(t,x)+\eta t,
&\text{ if }(t,x)\in [0,T']\times M,\\
 \infty,&\text{ otherwise. }\end{cases}
 \end{equation*}
verifies
\begin{equation}\label{acotacion}
\inf_{(0,T')\times M}\varphi < 0 \quad \text{ and } \inf_{\{0,T'\}\times M}\varphi >0.
\end{equation}

Since $v$ and $-u$ are lower semicontinuous in $[0,T)\times M$ and bounded in $[0,T']\times M$, the function
 $\varphi$ is  lower semicontinuous and bounded below.
Therefore, we can apply the Ekeland variational principle to $\varphi$ and any $\varepsilon>0$ (in the complete metric space $\mathbb R\times M$ with associated distance $D((r,y),(s,z))=|r-s|+d(y,z))$ in order to find
$(\overline{t},\overline{x}) \in [0,T']\times M$ such that
\begin{equation*}
\varphi (\overline{t},\overline{x})<0
\end{equation*}
and
\begin{equation*}
\varphi(t,x)\ge \varphi (\overline{t},\overline{x})-\varepsilon (|t-\overline{t}|+
d(x,\overline{x})), \, \, \text{ for all } (t,x)\in \mathbb R \times M.
\end{equation*}

\noindent
Thus $\varphi(t,x)+\varepsilon (|t-\overline{t}|+
d(x,\overline{x}))$ attains the minimum at $(\overline{t},\overline{x})$ and then
$0\in D^-\big(\varphi+\varepsilon (|\cdot-\overline{t}|+
d(\cdot,\overline{x}))\big)(\overline{t},\overline{x})$.
The boundedness conditions given in \eqref{acotacion}  yield $\overline{t}\in (0,T')$.

\medskip

 By assumption, let us assume that there is an open subset $(a,b)\times A\subset (0,T')\times M$ with $(\overline{t},\overline{x})\in (a,b)\times A$ and a constant
$K_{(\overline{t},\overline{x})}>0$
such that  $v$ is $K_{(\overline{t},\overline{x})}$--Lipschitz 
in  $(a,b)\times A$ (the other case is analogous).
Let $(U,\varphi)$ be a chart with $\overline{x}\in U\subset A$ satisfying the Palais
condition for $1+\overline{\varepsilon}$, where
$\overline{\varepsilon}=\min\{\varepsilon, \varepsilon K_{(\overline{t},\overline{x})}^{-1}\}$.

The set $(0,T')\times M$ is a Finsler manifold with the same smoothness properties as $M$, i.e. $(0,T')\times M$ is a $C^1$ Finsler manifold modeled over a Banach space with a $C^1$  Lipschitz  bump function.
Moreover, if $(U,\varphi)$ is the above chart  in $M$ with  $\overline{x}\in U$ satisfying the Palais
condition
for $1+\overline{\varepsilon}$, then    $(V, \phi)$  with
$V= (a,b)\times U$ and $\phi(t,x)=(t,\varphi(x))$ is a chart in $(0,T')\times M$
with $(\overline{t}, \overline{x})\in V$. In addition, this chart satisfies the Palais condition
for $1+\overline{\varepsilon}$ for the norms in the tangent space $T_{(t,x)}((0,T')\times M)$ defined
as  $||(r,v)||_{(t,x)}=|r|+||v||_{x}$. Notice that, in this case, the dual norm
in  $T_{(t,x)}((0,T')\times M)^*$ is $||(s,\Lambda)||_{(t,x)}^*=
\max\{|s|,||\Lambda||_{x}^*\}$. 

Let us recall that  there is a Lipschitz extension $\tilde{v}:\mathbb R \times M \to \mathbb R$
of the restriction $v|_V$,  there is   a lower semicontinuous extension $\tilde{u}:\mathbb R\times M\to \mathbb R$ of the function $-u:[0,T']\times M\to \mathbb R$ 
and $g(t,x)=\eta t + \varepsilon(|t-\overline{t}|+
d(x,\overline{x}))$ is Lipschitz in $\mathbb R\times M$.
Thus, by applying Proposition \ref{fuzzy} (the fuzzy rule for the subdifferential of the sum) to $\tilde{v}-\tilde{u}+g$, we  find $t_1,\,t_2, \, t_3 \in
(a,b)$, $x_1,x_2,x_3\in U$, $(\alpha_1,\Delta_1)\in D^{-}v(t_1,x_1)$,
 $(\alpha_2,\Delta_2)\in D^{-}(-u)(t_2,x_2)$  and $(\alpha_3,\Delta_3)\in D^{-}g(t_3,x_3)$ such that
\begin{enumerate}
\item[(i)] $|t_i-\overline{t}|<\varepsilon$ and $d(x_i,\overline{x})
<\varepsilon$ for $i=1,2,3$,
\item[(ii)] $|v(t_1,x_1)-v(\overline{t},\overline{x})|<\varepsilon$,
	$|u(t_2,x_2)-u(\overline{t},\overline{x})|<\varepsilon$ and
$|g(t_3,x_3)-g(\overline{t},\overline{x})|<\varepsilon$,
\item[(iii)]  $|\alpha_1+\alpha_2+\alpha_3|<\varepsilon$ and
$|||\Delta_1\circ d\varphi(x_1)^{-1}
+\Delta_2\circ d\varphi(x_2)^{-1}+\Delta_3\circ d
\varphi(x_3)^{-1}|||_{\overline{x}}<\varepsilon$, where $|||\cdot|||_{\overline{x}}$ is defined
as in the proof of Theorem \ref{estacionarias}, i.e. $|||w|||_{\overline{x}}=
||d\varphi^{-1}(\varphi(\overline{x}))(w)||_{\overline{x}} $ for $w\in X$, and
\item[(iv)] $\max\big\{|||\Delta_1\circ d\varphi(x_1)^{-1}|||_{\overline{x}}\,,\, 
|||\Delta_2\circ d\varphi(x_2)^{-1}|||_{\overline{x}}\big\}
\big(|t_1-t_2|+d(x_1,x_2)\big)<\varepsilon$.
\end{enumerate}
Let us write $\Lambda_1:=\Delta_1\in D^{-}_xv(t_1,x_1)=\pi_2(D^{-}v(t_1,x_1))$ where $\pi_2:\mathbb R\times TM^*\to TM^*$ is the canonical proyection over $TM^*$,
$\Lambda_2=-\Delta_2\in D^{+}_xu(t_2,x_2)=\pi_2( D^{+}_xu(t_2,x_2))$ and $\Lambda_3=\Delta_3\in D^{-}_xg(t_3,x_3)=
D^{-}(\varepsilon d(\cdot,\overline{x}))(x_3)$.
Notice that $(-\alpha_2, \Lambda_2)\in  D^{+}u(t_2,x_2) $.

\noindent The second inequality in (iii) yields
\begin{equation*}
\Big|\tresp{\Lambda_1\circ d\varphi(x_1)^{-1}}_{\overline{x}}
-\tresp{\Lambda_2\circ d\varphi(x_2)^{-1}}_{\overline{x}}\Big|\leq
\tresp{\Lambda_3\circ d\varphi(x_3)^{-1}}_{\overline{x}}
+\varepsilon.
\end{equation*}
The  function $\varepsilon d(\cdot,\overline{x})$ is $\varepsilon$-Lipschitz and thus
$||\Lambda_3||_{x_3}\le \varepsilon$.
Therefore,
\begin{equation*}
\tresp{\Lambda_3\circ d\varphi(x_3)^{-1}}_{\overline{x}}
\leq \|\Lambda_3\|_{x_3}\tresp{d\varphi(
x_3)^{-1}}_{\overline{x},x_3}=\|\Lambda_3\|_{x_3}(1+\varepsilon)<\varepsilon(1+\varepsilon),
\end{equation*}
where the norms $|||\cdot|||_{x,\overline{x}}$ and $|||\cdot|||_{\overline{x},x}$ for $x\in U$
are defined as in Theorem \ref{estacionarias}.
Also,
\begin{equation*}
\tresp{\Lambda_2\circ d\varphi(x_2)^{-1}}_{\overline{x}}
\geq \|\Lambda_2\|_{x_2}|||d\varphi(x_2)|||_{x_2,\overline{x}}^{-1}
\geq \|\Lambda_2\|_{x_2}(1+\overline{\varepsilon})^{-1}
\end{equation*}
and
\begin{equation*}
\tresp{\Lambda_1\circ d\varphi(x_1)^{-1}}_{\overline{x}}
\leq \|\Lambda_1\|_{x_1}\tresp{d\varphi(x_1)^{-1}}_{\overline{x},x_1}
\leq \|\Lambda_1\|_{x_1}(1+\overline{\varepsilon}).
\end{equation*}
Therefore,
\begin{equation*}
\|\Lambda_2\|_{x_2}(1+\overline{\varepsilon})^{-1}
-\|\Lambda_1\|_{x_1}(1+\overline{\varepsilon})<\varepsilon(2+\varepsilon).
\end{equation*}
Since $v$ is $K_{(\overline{t},\overline{x})}$ --Lipschitz in $V$, then
$||\Lambda_1||_{x_1}\le K_{(\overline{t},\overline{x})}$ and, by computing, we obtain
\begin{align*}
\|\Lambda_2\|_{x_2}- \|\Lambda_1\|_{x_1} &<\varepsilon(2+\varepsilon)
+\overline{\varepsilon}\|\Lambda_1\|_{x_1}
+\frac{\overline{\varepsilon}}{1+\overline{\varepsilon}} \|\Lambda_2\|_{x_2}
 \\ &\le \varepsilon(2+\varepsilon)
+ \overline{\varepsilon}\|\Lambda_1\|_{x_1}
+\overline{\varepsilon}\bigl( \varepsilon(2+\varepsilon)+(1+\overline{\varepsilon})
 \|\Lambda_1\|_{x_1}\bigr) \\
 &\le   \varepsilon(4+4\varepsilon+\varepsilon^2).
\end{align*}
In an analogous way we obtain $\|\Lambda_1\|_{x_1}- \|\Lambda_2\|_{x_2}
< \varepsilon(4+4\varepsilon+\varepsilon^2)$. 

Now, since $u$ is a viscosity subsolution and $v$ is a viscosity
supersolution of \eqref{E3} and the fact that $(\alpha_1,\Lambda_1)\in D^-{v}(t_1,x_1)$,
$(-\alpha_2,\Lambda_2)\in D^+{u}(t_2,x_2)$, we have
\begin{align*}
\alpha_1+H(t_1,x_1,\|\Lambda_1\|_{x_1})&\ge 0,\\
-\alpha_2+H(t_2,x_2,\|\Lambda_2\|_{x_2})&\le 0.
\end{align*}
Thus
\begin{equation*}
\alpha_1+\alpha_2+H(t_1,x_1,\|\Lambda_1\|_{x_1})-H(t_2,x_2,\|\Lambda_2\|_{x_2})\ge 0.
\end{equation*}
From condition (iii), we obtain
\begin{equation*}
-\alpha_3+\varepsilon + H(t_1,x_1,\|\Lambda_1\|_{x_1})-H(t_2,x_2,\|\Lambda_2\|_{x_2})\ge 0.
\end{equation*}
Since $\alpha_3\in D^-(\eta t+\varepsilon |t-\overline{t}|)(t_3)$, we have that
$\eta-\varepsilon\le \alpha_3\le \eta+\varepsilon$ and thus
\begin{equation} \label{HH}
-\eta+2\varepsilon + H(t_1,x_1,\|\Lambda_1\|_{x_1})-H(t_2,x_2,\|\Lambda_2\|_{x_2})\ge 0.
\end{equation}
Therefore, 
\begin{align*} 
-\eta \ge -2\varepsilon & - H(t_1,x_1,\|\Lambda_1\|_{x_1})+H(t_2,x_2,\|\Lambda_2\|_{x_2})
\\
  \ge -2\varepsilon  &
-\omega\big(|t_1-t_2|, d(x_1,x_2), \|\Lambda_1\|_{x_1}-
\|\Lambda_2\|_{x_2}\big) \\
&-C\max\big\{\|\Lambda_1\|_{x_1}\,,\|\Lambda_2\|_{x_2}\big\}\big(|t_1-t_2|+d(x_1,x_2)\big)
\\ \ge -2\varepsilon  &
-\omega\big(|t_1-t_2|, d(x_1,x_2), \|\Lambda_1\|_{x_1}-
\|\Lambda_2\|_{x_2}\big) \\
&-C(1+\varepsilon)\max\big\{|||\Lambda_1\circ d\varphi(x_1)^{-1}|||_{\overline{x}}\,,|||\Lambda_2 \circ d\varphi(x_2)^{-1}|||_{\overline{x}}\big\}\big(|t_1-t_2|+d(x_1,x_2)\big),
\end{align*}
where $\omega$ is the  function and $C\ge 0$ is the constant given in condition (A) for $H$. In addition,  $|t_1-t_2|<2\varepsilon$ and $d(x_1,x_2)<2\varepsilon$. Now, from the continuity of $\omega$ and condition (iv), we obtain  $H(t_1,x_1,\|\Lambda_1\|_{x_1})
 -H(t_2,x_2,\|\Lambda_2\|_{x_2})\to 0$ as $\varepsilon\to 0$ and thus  $-\eta\ge 0$, which is a contradiction.
\end{proof}

\begin{Rem} \begin{enumerate} \item  We say that a function $f:(0,\infty)\times M\to \mathbb R$ is 
$L$-Lipschitz in the second variable if
$|f(t,y)-f(t,z)| \le L d(y,z)$,  for all $(t,y), (t,z) \in (0,\infty)\times M$. 
 The assumptions on $u$ and $v$ in Theorem \ref{comparisonevolution} can be weakened in the following way:   $u$ is a viscosity subsolution of \eqref{E3} and $v$ is a viscosity
  supersolution of \eqref{E3}, for every $T>0$ both functions are bounded in $[0,T)\times M$,  for every $(t,x)\in (0,\infty) \times M$ either $u$ or $v$ is uniformly continuous in a neighborhood  of $(t,x)$, and finally  for every $(t,x)\in (0,\infty) \times M$  either $u$ or $v$ is Lipschitz in the second variable in a neighborhood  of $(t,x)$.

\item Let us  assume in the hypothesis of Theorem \ref{comparisonevolution} the additional condition: there is   $L>0$ such that either $u$ or $v$ is $L$-Lipschitz in the second variable in $(0,\infty)\times M$. Then,
 it is enough to assume that the Hamiltonian $H$ is uniformly continuous in $[0,\infty)\times M\times [0,R]$ for some $R>L$. 

\noindent
Let us consider the example $H:[0,\infty)\times M\times \mathbb R\to \mathbb R$, \,$H(t,x,m)=r(t,x)m$, where $r:[0,\infty)\times M\to \mathbb R$ is  a bounded and  uniformly continuous  function and the associated Hamilton-Jacobi equation
\begin{equation}\label{E3*}\tag{E3*}
\begin{cases}u_t(t,x)+r(t,x)||u_x(t,x)||_x=0, &   (t,x)\in (0,\infty)\times M,
\\u(0,x)=h(x), & x\in M.\end{cases}
\end{equation} The Hamiltonian $H$ is uniformly continuous in $[0,\infty)\times 
M \times[0,R]$ for every $R>0$. Let us denote by $\mathcal L$ the family of locally uniformly continuous functions $u:[0,\infty)\times M\to \mathbb R$ which are Lipschitz  in the second variable in $(0,\infty)\times M$
and bounded in $[0,T)\times M$ for every $T>0$.
Then, there is at most one function within $\mathcal L$ which is a viscosity solution of  equation \eqref{E3*}.

\item  It is worth noticing that Theorem \ref{comparisonevolution} holds (with few modifications in the proof) for the weaker condition (*) for $H$ given in \cite{DG} (see Remark \ref{remarksestacionarias}): a Hamiltonian $H$ of \eqref{E3}
verifies condition (*) if
\begin{equation*}
|H(t_1,x_1,r)-H(t_2,x_2,r)|\to 0 \text{ as }  (d(x_1,x_2)+|t_1-t_2|)(1+|r|) \to 0 
 \text{ uniformly on } t_1,t_2,r \in \mathbb  R, \, x_1, x_2 \in M,
\end{equation*}
\begin{equation*}
|H(t,x,r_1)-H(t,x,r_2)|\to 0 \text{ \ \ \ as \  \ \ }  |r_1-r_2|\to 0  \text{  \ \ \  uniformly on \  } 
 x \in M, \, t,r_1, r_2 \in \mathbb R.
\end{equation*}
\end{enumerate}
\end{Rem}


A few modifications of Theorem \ref{comparisonevolution} yield the following result on the monotonicity of the viscosity solutions.


\begin{Pro}\label{evolution-monotony}
Let $M$ be a complete  ${C}^1$ Finsler manifold modeled on a Banach space with
a $C^1$ Lipschitz bump function
 and let $H_1,\, H_2: M\times \mathbb R\to \mathbb{R}$ be two 
Hamiltonians of \eqref{E3} verifying condition (A) such that $H_1 \le H_2$.
Let us assume that  $v$ is a viscosity supersolution of \eqref{E3} with Hamiltonian $H_1$ and initial condition $v(0,x)=h_1(x)$ (for $x\in M$) and $u$ is a viscosity subsolution of \eqref{E3} with Hamiltonian $H_2$ and initial condition $u(0,x)=h_2(x)$ (for $x\in M$), where $h_1$ and $h_2$ are bounded and continuous on $M$,  and  $h_2\le h_1$. In addition, let us assume that for every $T>0$ the functions $u$ and $v$ are bounded in $[0,T)\times M$  and  for every $(t,x)\in (0,\infty) \times M$ either $u$ or $v$ is Lipschitz  in a neighborhood  of $(t,x)$. Then,
\begin{equation*}
\sup_{[0,\infty)\times M}(u-v)\le \sup_{[0,\infty)\times M}(H_2-H_1) + \sup_M\,(h_2-h_1).
\end{equation*}
\end{Pro}
Let us give an outline of  the proof of Proposition \ref{evolution-monotony} for completeness. Let us assume, by contradiction,  that 
\begin{equation*}
\inf_{[0,\infty)\times M}(v-u)<  \inf_{[0,\infty)\times M}(H_1-H_2) + \inf_M\,(h_1-h_2).
\end{equation*}
Let us consider the function $v-u-i$,  where $i:= \inf_M\,(h_1-h_2)$. Then  $\inf_{[0,\infty)\times M}(v-u-i)<  \inf_{[0,\infty)\times M}(H_1-H_2)\le 0$.  Notice that $(v-u-i)(0,x)\ge 0$ for all $x\in M$.
We can obtain  analogous inequalities for this function to the one given in \eqref{cotas} 
for $v-u$. In particular, equation \eqref{HH} becomes 
\begin{align*}
0 &\le-\eta+2\varepsilon + H_1(t_1,x_1,\|\Lambda_1\|_{x_1})-H_2(t_2,x_2,\|\Lambda_2\|_{x_2}) \\
&  = -\eta + 2\varepsilon + H_1(t_1,x_1,\|\Lambda_1\|_{x_1})- H_1(t_2,x_2,\|\Lambda_2\|_{x_2})  \\
& \qquad  +  H_1(t_2,x_2,\|\Lambda_2\|_{x_2}) - H_2(t_2,x_2,\|\Lambda_2\|_{x_2}) \\
& \le -\eta + 2\varepsilon + H_1(t_1,x_1,\|\Lambda_1\|_{x_1})- H_1(t_2,x_2,\|\Lambda_2\|_{x_2})  \\
& \qquad  +  \sup_{[0,\infty)\times M} (H_1 - H_2), 
\end{align*}
where $|t_1-t_2|<2\varepsilon$,\, $d(x_1,x_2)<2\varepsilon$, \,  $\left| \,\|\Lambda_1\|_{x_1}- \|\Lambda_2\|_{x_2}\right|
< \varepsilon(4+4\varepsilon+\varepsilon^2)$ and $C\max\big\{\|\Lambda_1\|_{x_1}\,,\|\Lambda_2\|_{x_2}\big\}\big(|t_1-t_2|+d(x_1,x_2)\big)<\varepsilon (1+\varepsilon)$. By letting $\varepsilon\to 0$,
property (A) for $H_1$ yields
$0\le -\eta +\sup_{[0,\infty)\times M} (H_1 - H_2)$, 
which is a contradiction because $\eta>0$ and $H_1\le H_2$.

\bigskip

Finally, let us give  existence results of viscosity subsolutions, supersolutions and solutions for Hamilton-Jacobi equations of the form \eqref{E3}. The proofs are analogous to those given in the preceding section. The first one is a straightforward consequence of Proposition
\ref{familias de subsoluciones}.

\begin{Cor}\label{existenceevolution1} Let $M$ be a   ${C}^1$ Finsler manifold modeled on a Banach space with
a $C^1$ Lipschitz bump function and an open subset  $\Omega$  of $M$. Let us consider the ${C}^1$ Finsler manifold $N=(0,\infty)\times M$ and the open subset $A=(0,\infty)\times \Omega$ of $N$, a continuous Hamiltonian 
$F:TA^* \times \mathbb R\to \mathbb R$ and a continuous function $h:\Omega \to \mathbb R$. Consider the Hamilton-Jacobi equation
\begin{equation}\label{HJEVOL}\tag{E4}
\begin{cases}F(t,x,u_t(t,x),u_x(t,x), u(t,x))=0, \quad (t,x)\in A,& \\
u(0,x)=h(x), \quad x\in \Omega.& \end{cases}
\end{equation}
 Assume that there are
 continuous functions $s_0,s_1: [0,\infty)\times \Omega \to \mathbb R$ with $s_0\le s_1$ and $s_0(0,x)=s_1(0,x)=h(x)$ for all $x\in \Omega$ such that $s_0$ and $s_1$ are respectively a viscosity subsolution and a viscosity supersolution of \eqref{HJEVOL}.
Let us consider the family
\begin{equation*}
\mathcal F =\{w:[0,\infty)\times \Omega \to \mathbb R: \, s_0\le w\le s_1 \text{ and }
w \text{ is a viscosity subsolution of } \eqref{HJEVOL} \}.
\end{equation*}
Let us define $u=\sup \mathcal F$. Then,   $u^*$ is a viscosity subsolution of \eqref{HJEVOL}
and $(u^*)_*$ is a viscosity supersolution of \eqref{HJEVOL}.
\end{Cor}

\begin{proof}

 First, let us recall that for a function $g:[0,\infty)\times \Omega \to \mathbb R$ and the restriction ${r}=g|_{A}$, we have
${r}^*(t,x)=g^*(t,x)$  and
${r}_*(t,x)=g_*(t,x)$ for $(t,x)\in A$. Also, recall that $N$ is a   ${C}^1$ Finsler manifold modeled on a Banach space with
a $C^1$ Lipschitz bump function.

Thus,
 the inequality $F(t,x,(u^*)_t(t,x),(u^*)_x(t,x),u^*(t,x))\le 0$ for all the superdifferentials of $u^*$ in $A$   is a consequence of
Proposition \ref{familias de subsoluciones} for the open subset $A$ of the Finsler manifold $N$.
For the initial condition, notice that $s_0\le u\le s_1$ and $s_0, s_1$ are continuous.
Therefore, $s_0\le u^*\le s_1$. In particular, $s_0(0,x)\le u^*(0,x)\le s_1(0,x)$ for all
$x\in \Omega$ and thus $s_0(0,x)= u^*(0,x)= s_1(0,x)=h(x)$ for all $x\in \Omega$.

 Analogously, $v=(u^*)_*$ is a supersolution:  The inequality $F(t,x,v_t(t,x),v_x(t,x), v(t,x))\ge 0$  for all the subdifferentials of $v$ in $A$  is a consequence of
Proposition \ref{familias de subsoluciones} for the open subset $A$ of the Finsler manifold $N$. The initial condition is obtained from the fact that
 $s_0\le u^*\le s_1$ and $s_0, s_1$ are continuous. Thus, $s_0\le (u^*)_*\le s_1$ and
 then $s_0(0,x)= (u^*)_*(0,x)= s_1(0,x)=h(x)$ for all $x\in \Omega$.
 \end{proof}

\begin{Cor} \label{existenciaFinsler2}
Let $M$ be a complete ${C}^1$ Finsler manifold modeled on a Banach space with
a $C^1$ Lipschitz bump function. Let $H:[0,\infty)\times M\times \Real \to \Real$ be the Hamiltonian of \eqref{E3}. Assume that $H$ verifies condition (A), the initial condition $h:M\to \mathbb R$ is $L$-Lipschitz and bounded, and there are constants $K_0,K_1\in \mathbb R$ such that
\begin{equation*}K_0=\inf\{H(t,x,m):\,(t,x)\in (0,\infty)\times M,\,|m|\le L \}
 \end{equation*}
and
\begin{equation*} K_1=\sup\{H(t,x,m):\,(t,x)\in (0,\infty)\times M, \,|m|\le L\}.
\end{equation*}
Let us define
\begin{align*}
\mathcal F=\{w:&[0,\infty)\times M\to \mathbb R:\,
w \text{ is a subsolution of  \eqref{E3}  and } \\ 
& \qquad -K_1 t +h(x)\le w(t,x)\le -K_0t +h(x) \text{ for } (t,x)\in (0,\infty)\times M\}\end{align*}
 and $u=\sup \mathcal F$. Then,   $u^*$ is a viscosity subsolution and  $(u^*)_*$  is a viscosity supersolution  \eqref{E3}. Moreover,

 \begin{enumerate}
 \item if $u^*$ is continuous, then
 $u^*=(u^*)_*$ and $u^*$ is a viscosity solution of  \eqref{E3};
 \item if $u^*$ is locally Lipschitz, then
 $u^*$ is the unique viscosity solution of  \eqref{E3} which is bounded in $[0,T)\times M$ for every $T>0$.
 \end{enumerate}

 \end{Cor}
 \begin{proof} Notice that $s_0(t,x)=-K_1 t +h(x)$, for $(t,x)\in [0,\infty)\times M$ is a viscosity subsolution of \eqref{E3} and $s_1(t,x)=-K_0 t +h(x)$, for $(t,x)\in [0,\infty)\times M$ is a viscosity supersolution of \eqref{E3}.  Corollary \ref{existenceevolution1} yields $u^*$ and $(u^*)_*$ are respectively a viscosity subsolution and a viscosity supersolution  of \eqref{E3}.

 If, in addition, we assume that $u^*$ is continuous, then by the definition of lower semicontinuous envelope, $u^*=(u^*)_*$ and therefore it is a viscosity solution of \eqref{E3}.

 If, in addition, we assume that $u^*$ is locally Lipschitz, the inequality 
$s_0\le u^*\le s_1$
 in $[0,\infty)\times M$  yields the boundedness of
 $u^*$ in $[0,T)\times M$ for all $T>0$. Therefore, we can apply the comparison result given in Theorem \ref{comparisonevolution} to obtain that $u^*$ is the unique viscosity solution
 of \eqref{E3}  which is bounded on $[0,T)\times M$ for all $T>0$. Thus, if there exists  $w$  a different viscosity solution of \eqref{E3}, then there is $T_0>0$ such that $w$ is not  bounded in $[0,T_0)\times M$).
 \end{proof}

\noindent {\bf Acknowledgments.} The authors warmly thank the anonymous reviewer for providing very helpful comments and references.

\end{document}